\documentclass[11pt]{article}
\usepackage{amsfonts}
\usepackage{mathrsfs}

\usepackage[latin1]{inputenc}
\usepackage{amsmath,amssymb}
\usepackage{latexsym}
\usepackage[active]{srcltx}
\usepackage[
bookmarks=true,         
bookmarksnumbered=true, 
colorlinks=true, pdfstartview=FitV, linkcolor=blue, citecolor=blue,
urlcolor=blue]{hyperref}

\newtheorem{theorem}{Theorem}[section]

\newtheorem{lemma}[theorem]{Lemma}
\newtheorem{proposition}[theorem]{Proposition}

\newtheorem{remark}{Remark}
\numberwithin{equation}{section}
\def\R{{\mathbb R}}
\def\E{{{\mathbb E}\,}}
\def\P{{\mathbb P}}

\def\Z{{\mathbb Z}}

\def\N{{\mathbb N}}

\def\al{{\alpha}}

\def\eps{\varepsilon}

\def\sgn{{\mathop {{\rm sgn\, }}}}
\def\Var{{\mathop {{\rm Var\, }}}}

\def\square{{\vcenter{\vbox{\hrule height.3pt
        \hbox{\vrule width.3pt height5pt \kern5pt
           \vrule width.3pt}
        \hrule height.3pt}}}}

\def\tlint{{- \kern-0.85em \int \kern-0.2em}}
\def\dlint{{- \kern-1.05em \int \kern-0.4em}}

\def\si{\sigma}

\def\be{{\beta}}

\def\si{{\sigma}}

\def\al{{\alpha}}

\def\be{{\beta}}

\def\ga{{\gamma}}

\def\si{{\sigma}}

\def \eref#1{\hbox{(\ref{#1})}}

\def \eref#1{\hbox{(\ref{#1})}}

\def\si{{\sigma}}
\def\al{{\alpha}}

\newenvironment{proof}[1][Proof]{\noindent\textit{#1.} }{\hfill \rule{0.5em}{0.5em}}

\author{Yudan Xiong, Fangjun Xu and Jinjiong Yu}

\begin{document}

\title{Limit theorems for functionals of linear processes in critical regions}
\date{\today}
\maketitle

\begin{abstract}

Let $X=\{X_n: n\in\mathbb{N}\}$ be the linear process defined by $X_n=\sum^{\infty}_{j=1} a_j\varepsilon_{n-j}$,
where the coefficients $a_j=j^{-\beta}\ell(j)$ are constants with $\beta>0$ and $\ell$ a slowly varying function, and the innovations $\{\varepsilon_n\}_{n\in\Z}$ are i.i.d.\ random variables belonging to the domain of attraction of an $\alpha$-stable law with $\al\in(0,2]$.
Limit theorems for the partial sum $
S_{[Nt]}=\sum^{[Nt]}_{n=1}[K(X_n)-\mathbb{E}K(X_n)]$ with proper measurable functions $K$ have been extensively studied, except for two critical regions: I.~$\al\in(1,2),\beta=1$ and II.~$\al\be=2,\be\geq1$.
In this paper, we address these open scenarios and identify the asymptotic distributions of $S_{[Nt]}$ under mild conditions.
\end{abstract}


\section{Introduction}

A linear process $X=\{X_n: n\in\mathbb{N}\}$ is defined by
\begin{align} \label{lp}
X_n=\sum^{\infty}_{j=1} a_j\varepsilon_{n-j}, 
\end{align}
where $a_j$'s are constant coefficients, and the innovations $\{\varepsilon_n\}_{n\in\Z}$ are i.i.d.\ random variables belonging to the domain of attraction of an $\alpha$-stable law with $\al\in(0,2]$.
The linear processes \eqref{lp} form a fundamental class of time series models, and are widely applied in diverse fields such as physics, economics, finance, network engineering, etc. For a comprehensive overview, see \cite{BFGK13, DOT03} and references therein.
Among the extensive body of literature on linear processes, significant attention has been given to studying the functional limit theorems for the partial sum
\begin{equation}
	S_{[Nt]}:=\sum^{[Nt]}_{n=1}\big[K(X_n)-\mathbb{E}K(X_n)\big],
\end{equation}
where $K$ is some suitable measurable functions and $[Nt]$ is the integer part of $Nt$.

We assume the coefficients in (\ref{lp}) take the form 
\begin{equation}\label{aj}
	a_j=j^{-\beta} \ell(j),
\end{equation}
where $\beta>0$ and $\ell:(0,\infty)\to(0,\infty)$ is a slowly varying function at $\infty$, i.e., $\lim_{t\to\infty}\ell(ct)/\ell(t)=1$ for any constant $c>0$.
For the innovations, we moreover assume that they are symmetric for $\alpha = 1$ and centered for $\alpha>1$.
By \cite[Theorems 2.6.1 and 2.6.2]{IL}, there exist nonnegative constants $\sigma_1, \sigma_2\geq 0$ with $\sigma_1+\sigma_2>$ such that as $x\to\infty$,
\begin{align} \label{h}
	\mathbb{P}(\varepsilon_1\leq -x)&=(\sigma_1+o(1))x^{-\alpha} h(x)\quad \text{and}\quad \mathbb{P}(\varepsilon_1>x)=(\sigma_2+o(1))x^{-\alpha} h(x),
\end{align}
where $h(x)$ is a positive slowly varying function at $\infty$. Equivalently, the characteristic function $\phi_{\eps}(u)$ of $\eps_1$ satisfies (see, e.g., \cite{IL})
\begin{align} \label{char}
	\ln \phi_{\eps}(u)=\ln \mathbb{E} [e^{\iota u \varepsilon_1}] =-(\sigma+o(1))|u|^{\alpha}H(|u|^{-1})(1-\iota D \sgn(u))\quad \text{as}~u\to0,
\end{align}
where $\iota=\sqrt{-1}$, $\sigma=(\sigma_1+\sigma_2)\Gamma(|\alpha-1|)\cos(\frac{\pi\alpha}{2})$, $D=\frac{\sigma_1-\sigma_2}{\sigma_1+\sigma_2}\tan(\frac{\pi\alpha}{2})$ in case $\alpha \neq 1$; $\sigma=(\sigma_1+\sigma_2)\frac{\pi\alpha}{2}$, $D=0$ in case $\alpha = 1$, and 
\begin{align} \label{H}
	H(t)=\left\{\begin{array}{ll}
		\displaystyle h(t) & \displaystyle \text{for}\; \alpha\in (0,2), \\ [5pt]
		-\int^t_0 s^2 d\frac{h(s)}{s^2} & \displaystyle \text{for}\; \alpha=2
	\end{array}\right.
\end{align}
is also a slowly varying function.

By the three series theorem, the linear process $X=\{X_n: n\in\mathbb{N}\}$ in \eref{lp} is well-defined if and only if $\sum^{\infty}_{j=1} a_j^{\alpha}H(a_j^{-1})<\infty$. 
Recalling $a_j$ from (\ref{aj}), the summand is $(-\alpha\beta)$-regularly varying in $j$, and therefore the linear process $X$ exists if $\alpha\beta>1$. 
Based on this condition, $X$ is said to exhibit long memory (resp.\ short memory) if $\sum_{j=1}^{\infty}\sqrt{1-|\phi_{\eps_1}(a_j\lambda)|^2}=\infty$ (resp.\ $<\infty$) for all $\lambda$ close enough to 0, see \cite[Definition~2.1]{SSX}.
Expanding the characteristic function $\phi_{\eps}(u)$ in (\ref{char}) at $0$, we see that $X$ is a long memory linear process if 
\begin{equation}\label{lmcri}
	\sum_{j=1}^{\infty}a_j^{\frac{\al}{2}}H^{\frac{1}{2}}(a_j^{-1})=\infty,
\end{equation} and is a short memory linear process otherwise.

Limit theorems for $S_{[Nt]}$ have been established under various mild conditions.
Let $A_N$ denote the scaling factor such that $S_{[Nt]}/A_N$ converges.
\begin{itemize}
	\item In the short memory region $\al\be>2$, with $A_N=\sqrt{N}$, $S_{[Nt]}/A_N$
	converges in the process level to a Brownian motion \cite{Hsing}, and the conditions for convergence of marginal distributions were later relaxed in \cite{PT03}.

	\item In the long memory region $1<\alpha\beta<2$, the limiting behavior depends on $\be$.
	Assume that the slowly varying functions $\ell$ in (\ref{aj}) and $H$ in (\ref{H}) are constants:
	\begin{itemize}
		\item [$\bullet$] For $1/\al<\be<1$, \cite{KS01} showed that $A_N=N^{1-\be+1/\al}$, and $S_{N}/A_N$
	converges to an $\al$-stable distribution. 
		\item [$\bullet$] For $1<\be<2/\al$, \cite{Hon09, sur02} showed that $A_N=N^{1/\al\be}$, and $S_{[Nt]}/A_N$ converges to an $\al\be$-stable process.
	\end{itemize} 
	Extensions of these results for general $\ell$ and $H$ were achieved in \cite{LXX}, introducing slowly varying corrections to $A_N$ without changing the limits. 
\end{itemize}
However, two critical regions remain widely open:

{\hspace{1pt}I.} $\al\in(1,2)$ and $\beta=1$; \quad II. $\al\be=2$ and $\be\geq1$.

{\noindent}Only for a special choice $K(x)=x$ and $\beta=1$, limit theorems for $S_{[Nt]}$ have been previously established in \cite{PSXY} and \cite{X}. 
In this paper, we address these two critical cases, deriving the asymptotic behavior of $S_{[Nt]}$ under mild conditions. While some techniques from the works cited above remain applicable, new ideas are required for analyzing long memory linear processes in the critical regions, particularly for Theorems~\ref{thm1} and \ref{thm4} presented below.

Before stating our main results, we first make the following two assumptions:
\begin{enumerate}
	\item[({\bf A1})] $K$ is a real-valued integrable and square integrable function on $\mathbb{\R}$.
	\item[({\bf A2})] There exist strictly positive constants $c$ and $\delta$ such that for all $u\in\mathbb{R}$,
	$|\phi_{\varepsilon}(u)|\leq \frac{c}{1+|u|^{\delta}}$.
\end{enumerate}
The assumption ({\bf A2}) and {$a_j= j^{-\beta}\ell(j)$} imply (e.g., see \cite{LXX}) that
\begin{enumerate}
	\item [(i)] the linear process $X=\{X_n: n\in\mathbb{N}\}$ defined in \eref{lp} has a bounded and differentiable probability density function $f(x)$;
	\item [(ii)] the characteristic function $\phi(u)$ of $X_n$ decays to 0 faster than any polynomial as $|u|\to\infty$. Moreover, ({\bf A2}) implies that there exists $m\in\mathbb{N}$ such that $|\phi_{\varepsilon}(u)|^m$ is less than a constant multiple of $\frac{1}{1+|u|^4}$.
\end{enumerate}
To facilitate the analysis, we introduce a useful auxiliary function $K_\infty(x):=\E[K(X_1+x)]$, which has been utilized in pioneer works \cite{HH96,HH97}.
Under the assumption ({\bf A2}), $K_\infty$ is differentiable and 
\begin{equation}\label{Kinfd}
	K_\infty'(0):= - \int_\R K(x) d f(x).
\end{equation}
To ease notations, for any constants $\ga,\kappa>0$ and slowly varying function $G$, we denote $1/\ga$-regularly varying functions
\begin{equation}\label{Rsim}
	R_{\ga,G}(x):=x^{\frac{1}{\ga}}G^{\frac{1}{\ga}}(x)
	\quad\text{and}\quad
	R_{\ga,G,\kappa}(x):=x^{\frac{1}{\ga}}G^{\frac{1}{\ga}}(x^{\frac{1}{\kappa}}).
\end{equation}

{\noindent\bf Region I. $\al\in(1,2)$ and $\be=1$.}\\
In this region, $a_j^{\al/2}H^{1/2}(a_j^{-1})$ is $(-\al/2)$-regularly varying in $j$ and is not summable, thereby the long memory criterion (\ref{lmcri}) holds.
To state our first main result, we introduce the $\alpha$-stable process $Z^{\alpha}=\{Z^{\alpha}_t: t\geq 0\}$ which has the characteristic function
\begin{align} \label{stablep}
	\mathbb{E}[e^{\iota u Z^{\alpha}_t}]=\exp\big(-t\sigma |u|^{\alpha}(1-\iota D \sgn u)\big),
\end{align}
where $\sigma$ and $D$ are the same as in (\ref{char}).
Let 
\begin{equation}\label{L}
	L(N):=\sum_{j=1}^{N}a_j= \sum_{j=1}^{N}j^{-1}\ell(j)
\end{equation}
denote the partial sum of coefficients. By Karamata's theorem \cite[Theorem~1.6.5]{BGT89}, $L(N)$ is a slowly varying function satisfying $\lim_{N\to\infty}\ell(N)/L(N)=0$. For $x\notin\N$, $L(x)$ is defined by linear interpolation. By \cite[Theorem~1.5.13]{BGT89}, there exists a unique slowly varying function $H_{\alpha}$, up to asymptotic equivalence, such that as $N\to\infty$,
\begin{align} \label{halpha}
	H(N_{\al,H_\al})\sim H_{\alpha}(N),
\end{align}
where the simplified notation $N_{\al,H_\al}$ refers to the $(1/\al)$-regularly varying in (\ref{Rsim})
\begin{equation}\label{Nsim}
	N_{\al,H_\al}:=R_{\al,H_\al}(N)=N^{\frac{1}{\alpha}}H^{\frac{1}{\alpha}}_{\alpha}(N).
\end{equation}
\begin{theorem} \label{thm1} 
	For $1<\alpha<2$ and $\beta=1$, if assumptions  $(\bf A1)$ and $(\bf A2)$ hold, then as $N\to\infty$,
	\[
	\big\{A^{-1}_N S_{[Nt]}: t\geq 0\big\} \overset{\rm f.d.d.}{\longrightarrow} \big\{K'_\infty(0)\hspace{1pt} Z^{\alpha}_t:\; t\geq 0 \big\},
	\]
	where the scaling factor $A_N$ is given by
	\begin{equation}\label{sca1}
		A_N:=\inf\{x>0: x-N_{\al,H_\al}\big(L(N)-L(x)\big)\geq 0\},
	\end{equation}
	$K'_\infty(0)$ is the constant in \eqref{Kinfd}, $Z^{\alpha}=\{Z^{\alpha}_t: t\geq 0\}$ is the $\alpha$-stable process specified by \eqref{stablep}, 
	and $\overset{\rm f.d.d.}{\longrightarrow}$ denotes weak convergence in finite dimensional distributions.
\end{theorem}
\begin{remark}\label{rem1}
	{\rm (1)} It is easy to see that $N_{\al,H_\al}\big(L(N)-L(A_N)\big)\leq A_N<N_{\al,H_\al}\big(L(N)-L(A_N-1)\big)+1$.

	{\rm(2)} It is helpful to recall three typical examples of slowly varying functions:
	$$\setlength\abovedisplayskip{5pt}\setlength\belowdisplayskip{5pt}
		(i)~\ln\ln x, \qquad (ii)~(\ln x)^\kappa \quad(\kappa\in\mathbb{R}), \qquad (iii)~e^{(\ln x)^\kappa}\quad (0<\kappa<1).
	$$
	For the increasing slowly varying function $L$, if it is of the second or third type, then it follows that $L(N)-L(A_N)\sim L(N)$. Thereby, the scaling factor  $A_N\asymp N_{\alpha,H_{\alpha}}L(N)=N^{\frac{1}{\alpha}}H_{\alpha}^{\frac{1}{\alpha}}L(N)$, where $f_N\asymp g_N$ denotes that $f_N$ and $g_N$ have the same order as $N\to\infty$.
	This aligns with Theorem~\ref{thm4} when $\al=2,\be=1$.
	
	{\rm(3)} If $K'_{\infty}(0)=0$, then a nontrivial scaling limit of $S_{[Nt]}$ has been obtained in \cite[Theorem~1.3]{LXX}.
\end{remark}

{\noindent\bf Region II. $\al\be=2$ and $\be\geq1$, long memory.}\\
For $\al\be=2$, the series $\sum_{j=1}^{\infty}a_j^{\frac{\al}{2}}H^{\frac{1}{2}}(a_j^{-1})$ in the long/short memory criterion is an infinite sum with $(-1)$-regularly varying summand, which can be either divergent or convergent depending on $\ell$ in (\ref{aj}) and $H$ in (\ref{char}).
Though $S_{[Nt]}$ converges to a Brownian motion in both cases, the components of $S_{[Nt]}$ that contributes to the limit differ significantly, as do the scaling factors.
We first focus on the long-memory case, where it is necessary to further distinguish between the case on the critical curve $\al\be=2,\be>1$ and the case at the critical point $\al=2,\be=1$.

According to the Karamata representation \cite[Theorem~1.3.1]{BGT89}, the slowly varying function $\ell$ can be expressed as 
$\ell(x)=\sigma_0(x) e^{\int^x_1 \frac{\eta(t)}{t}}dt$, where $\lim_{x\to\infty}\sigma_0(x)=\sigma_0\neq 0$ and $\lim_{t\to\infty}\eta(t)=0$.
To avoid extra tedious analysis on the slowly varying functions, we impose the following technical condition on $\ell$ to guarantee that it is regular enough and does not diverge too rapidly. 
\begin{enumerate}
	\item[({\bf T1})] The slowly varying function $\ell$ has the Karamata representation $
		\ell(x)=\sigma_0 \, e^{\int^x_1 \frac{\eta(t)}{t}}dt$
	where $\sigma_0>0$, and $|\eta(t)|\leq (\ln t)^{-\ga}$ for some $\ga>1/2$ eventually (for sufficiently large $t$).
\end{enumerate}
Notably, functions such as $e^{(\ln x)^{\ga}}$ with $1/2\leq \ga< 1$ does not satisfy this condition.
Assuming ({\bf T1}), we observe that $R_{\be,\ell,\be}(x)=x^{\frac{1}{\be}} \ell^{\frac{1}{\be}}(x^{\frac{1}{\be}})$ is continuous and strictly increasing on $(A,\infty)$ for some $A>0$.
Thus, it has a continuous inverse
\begin{equation}\label{RInv}
	R_{\be,\ell,\be}^{\leftarrow}(x):=\inf\{s>A: R_{\be,\ell,\be}(s)\geq x \}.
\end{equation}
When there is no ambiguity, we write $R^{\leftarrow}_{\beta}(x)$ for $R_{\be,\ell,\be}^{\leftarrow}(x)$ for simplicity.
Since $R_{\be,\ell,\be}$ is $(1/\be)$-regularly varying, $R^{\leftarrow}_{\beta}$ is $\be$-regularly varying by \cite[Theorem~1.5.12]{BGT89}.
Consequently, the function 
\begin{equation*}
	\overline{H}(x):=\int_0^x t \hspace{1pt} \frac{h(R^{\leftarrow}_{\beta}(t))}{(R^{\leftarrow}_{\beta}(t))^{\alpha}}\, dt
\end{equation*}
is also slowly varying. 
As in (\ref{halpha}), up to asymptotic equivalence, there exists a unique slowly varying function $\overline{H}_2(N)$ such that
$\overline{H}(N_{2,\overline{H}_2})\sim \overline{H}_2(N)$, where $N_{2,\overline{H}_2}=N^{\frac{1}{2}}\overline{H}^\frac{1}{2}_2(N)$ (recall (\ref{Nsim})).
\begin{theorem} \label{thm2} 
	For either {\rm (a)} $\al\be=2$ and $\be>1$, or {\rm (b)} $\al=2,\be=1$ and $K'_\infty(0)=0$,
	if assumptions {\rm({\bf A1})}, {\rm({\bf A2})} and {\rm({\bf T1})} hold and $\lim_{x\to\infty}\overline{H}(x)=\infty$, then, as $N\to\infty$,
\[
		\big\{ A_N^{-1} S_{[Nt]}: t\geq 0\big\} \overset{\rm f.d.d.}{\longrightarrow} \big\{ \gamma W_t:\; t\geq 0 \big\},
		\]
		where $A_N= N^{\frac{1}{2}}\overline{H}^\frac{1}{2}_2(N)$, $(W_t)_{t\geq0}$ is the standard Brownian motion, and
		${\gamma=\sqrt{2(\gamma_1+\gamma_2)}}$ with 
\[
\left\{\begin{array}{l}
\displaystyle \gamma_1=\sigma_2 |C^+_K|^{2} 1_{\{C^+_K<0\}}+\sigma_1 |C^-_K|^{2} 1_{\{C^-_K<0\}}, \\ [10pt]
\displaystyle \gamma_2=\sigma_2 |C^+_K|^{2} 1_{\{C^+_K>0\}}+\sigma_1 |C^-_K|^{2} 1_{\{C^-_K>0\}}
\end{array}\right.
\]
and $C^{\pm}_K:= \int_{0}^{\infty} (K_{\infty}(\pm t^{-\beta}) - K_{\infty}(0))\, dt$.
\end{theorem}

\begin{remark}\label{rem3}
	Assumptions of Theorem~\ref{thm2} do not guarantee that the linear process $X=\{X_n: n\in\mathbb{N}\}$ has long memory. If in Theorem~\ref{thm2} we assume that $\sum_{j=1}^{\infty}j^{-1}((j^2a_j^\alpha h(a_j^{-1}))\wedge1)=\infty$ instead of $\lim_{x\to\infty}\overline{H}(x)=\infty$, then by the definition of $a_j$ and (\ref{H}),
	\begin{align*}
	\sum_{j=1}^{\infty}a_j^{\frac{\al}{2}}H^{\frac{1}{2}}(a_j^{-1}) = \sum_{j=1}^{\infty} j^{-1}\ell^{\frac{\al}{2}}(j) H^{\frac{1}{2}}(a_j^{-1}) \geq c_1\sum_{j=1}^{\infty}j^{-1}((\ell^\alpha(j) h(a_j^{-1}))\wedge1)=\infty
	\end{align*}
and thus the linear process exhibits long memory. Moreover, $\lim_{x \rightarrow \infty} \ell(x \ell^{-\frac{1}{\beta}}(x))/\ell(x)=1$ by Lemma \ref{lem42}. Set $x = y^{\beta} \ell^{-1}(y)$. As $y \rightarrow \infty$, $
	R_{\beta, \ell, \beta}(x) = x^{\frac{1}{\beta}} \ell^{\frac{1}{\beta}}(x^{\frac{1}{\beta}}) = y \ell^{-\frac{1}{\beta}}(y) \ell^{\frac{1}{\beta}}(y \ell^{-\frac{1}{\beta}}(y)) \sim y$. Then as $x \rightarrow \infty$,	$R^{\leftarrow}_{\beta}(x) \sim x^{\beta} \ell^{-1}(x)$. Therefore, as $t \rightarrow \infty$,
	\begin{align*}
		\overline{H}(t)=-\int^t_0 x \frac{h(R^{\leftarrow}_{\beta}(x))}{(R^{\leftarrow}_{\beta}(x))^{\alpha}} dx  \sim \int^t_0 x \frac{h(x^{\beta} \ell^{-1}(x))}{x^{\al\beta} \ell^{-\al}(x)} dx =\int_{0}^{t} \frac{\ell^{\al}(x) h\big(x^{\beta} \ell^{-1}(x)\big)}{x}  dx\to\infty.
	\end{align*}
\end{remark}

The proof of Theorem~\ref{thm2} (ii) is not new in nature. The approach is essentially the same as the one for \cite[Theorems~1.2, 1.3]{LXX}.
However, in general  when $K'_\infty(0)\neq0$, this approach cannot cover the case $\al=2,\be=1$.
To address this critical case, we develop a method under the extra assumption as below.
Recall the slowly varying functions $L$ from (\ref{L}) and $h$ from (\ref{h}).
\begin{enumerate}
	\item[({\bf T2})] For $L$ and $h$, there exist continuous functions $g_L$ and $g_h$ such that for any $\lambda\in(0,1)$,
	\begin{equation}
		\lim\limits_{x\to\infty}\frac{L(e^{\lambda x})}{L(e^x)}=g_L(\lambda)<1,\quad 
		\lim\limits_{x\to\infty}\frac{h(e^{\lambda x})}{h(e^x)}=g_h(\lambda)>0.
	\end{equation}
\end{enumerate}
We point out that for the three types of slowly varying functions in Remark \ref{rem1} (2), $L(x)=\ln\ln x$ or $h(x)=e^{(\ln x)^\kappa}$ does not satisfy ({\bf T2}).

\begin{theorem} \label{thm4} 
	For $\alpha=2$ and $\beta = 1$, if assumptions {\rm({\bf A1})}, {\rm({\bf A2})}, {\rm({\bf T1})} and {\rm({\bf T2})} hold, then {\rm (i)} The linear process exhibits long memory, i.e., $\sum_{j=1}^{\infty}a_jH^{\frac{1}{2}}(a_j^{-1})=\infty$. {\rm (ii)} Let $A_N= 	N^{\frac{1}{2}}H_{2}^{\frac{1}{2}}(N) L(N)$. As $N\to\infty$,
		\[
		\big\{ A_{N}^{-1}S_{[Nt]}: t\geq 0\big\} \overset{\rm f.d.d.}{\longrightarrow} \big\{ \gamma W_t:\; t\geq 0 \big\},
		\]
	where $(W_t)_{t\geq0}$ is the standard Brownian motion, and $\gamma= \sqrt{c_{L, h} (\si_{1}+\si_{2})}\, |K'_\infty(0)|$ with constants $\si_1$ and $\si_2$ given in (\ref{h}) and $
c_{L,h} =\int_{0}^{\frac{1}{2}} (1-g_{L}(y))^{2} g_{h}(y) dy/\int_{0}^{\frac{1}{2}} g_{h}(y) dy$.
\end{theorem}

{\noindent\bf Region II. $\al\be=2$ and $\be\geq1$, short memory.}\\
For $\al\be>2$, the linear process has short memory, and \cite{Hsing} showed that a central limit theorem holds for $S_{[Nt]}/\sqrt{N}$ by applying a truncation technique to $X_n$.
Later, \cite{PT03} generalised the method to include the short memory linear process when $\al\be=2$ and $\be>1$, under the assumption that $K$ is bounded and has bounded second derivative.
Based on the same truncation technique and the Fourier transform, we can improve the result by removing the differentiability condition on $K$.
\begin{theorem} \label{thm3}
	For $\al\be=2$ and $\be\geq1$, assume that {\rm(\bf A1)} and {\rm(\bf A2)} hold.
	If the linear process has short memory, i.e., $\sum_{j=1}^{\infty}
	a_{j}^{\frac{\al}{2}} H^{\frac{1}{2}}(a_{j}^{-1})<\infty$, then the limit $
	\theta^{2} = \lim_{N \rightarrow \infty} N^{-1} \Var(S_{N})$ exists and is finite, and as $N\to\infty$,
	\[
	\big\{ N^{- \frac{1}{2}}S_{[Nt]}: t\geq 0\big\} \overset{\rm f.d.d.}{\longrightarrow} \{ |\theta| W_t:\; t\geq 0 \big\},
	\]
	where $(W_t)_{t\geq0}$ is the standard Brownian motion.
\end{theorem}

\begin{remark}\label{rem2}
	Comparing the scaling factors $A_N$ in Theorem~\ref{thm3} and in Theorems~\ref{thm1}, \ref{thm2} and \ref{thm4}, we see that $\sqrt{N}\asymp A_N$ in the short memory region, while $\sqrt{N}=o(A_N)$ in the long memory region.
\end{remark}

{\noindent\bf Organization.}
The remaining of the paper is devoted to proofs. 
In Section~\ref{S:app}, we discuss the approximations of $S_N$ in the long and short memory regions using two distinct quantities tailored to each region.
Building on these approximations, we proceed to prove Theorems~\ref{thm1}-\ref{thm3} sequentially in Section~\ref{S:thm1}-\ref{S:thm3}.

\section{Approximations of $S_N$}\label{S:app}

The observation in Remark~\ref{rem2} that the scaling factor $A_N$ satisfies $\lim_{N\to\infty}A_N/\sqrt{N}=\infty$ (resp.\ $\limsup_{N\to\infty}A_N/\sqrt{N}<\infty$) in the long (resp.\ short) memory region has deeper underlying reasons.
These distinctions raise from two distinct approaches for approximating $S_N$, as described below.

{\noindent\bf Long memory region.}\\
In the linear process, the coefficients $a_j$ and innovations $\eps_n$ are intertwined, making it challenging to directly derive limit theorems for the partial sum $S_N$.
In the long memory region, a now standard approximation for $S_N$ (see \cite{LXX,sur02}) is given by
\begin{equation}\label{T}
	T_N = \sum_{n=1}^{N}\sum_{j=1}^{\infty} \big[K_\infty(a_j\eps_{n-j})-\E K_\infty(a_j\eps_{n-j})\big] \quad\text{with~}K_\infty(x)=\E[K(X_1+x)].
\end{equation}
Heuristically, the term $K_\infty(a_j\eps_{n-j})-\E K_\infty(a_j\eps_{n-j})$ in $T_N$ can be replaced by $K_\infty(a_j\eps_{n-j})-K_\infty(0)$, which is approximately $K'_\infty(0)a_j\eps_{n-j}$ by a first-order Taylor expansion.
This linearization simplifies the analysis of $T_N$ compared to $S_N$. The next lemma is useful when we utilize the Fourier transform.

	\begin{lemma}  \label{lemb}
		Suppose that $\varepsilon$ is in the domain of attraction of an $\alpha$-stable law with $\alpha\in (0,2]$,  $\varepsilon$ is symmetric for $\alpha=1$ and $\mathbb{E}\, \varepsilon=0$ for $\alpha\in (1,2]$, and $\phi_{\varepsilon}(\lambda)$ is the characteristic function of $\varepsilon$. Then, for any $\alpha'\in (0,\alpha)$, there exist positive constants $c_{\alpha,1}$, $c_{\alpha',1}$ and $c_{\alpha',2}$ such that
		\begin{align*}  
			1-|\phi_{\varepsilon}(\lambda)|^2=\E\big|e^{ \iota \lambda\varepsilon}-\phi_{\varepsilon}(\lambda)\big|^2
			\leq c_{\alpha,1} \big(|\lambda|^{\alpha}H(|\lambda|^{-1}) \wedge 1\big)
			\leq c_{\alpha',1}\big(|\lambda|^{\alpha'}\wedge 1\big)
		\end{align*}
		and 
		\begin{align*}  
			|1-\phi_{\varepsilon}(\lambda)|\leq c_{\alpha,1} \big(|\lambda|^{\alpha}H(|\lambda|^{-1}) \wedge 1\big)\leq c_{\alpha',2}\big(|\lambda|^{\alpha'}\wedge 1\big).
		\end{align*}
	\end{lemma}
	
	\begin{proof} This follows from (\ref{char}), \cite[Theorem 2.6.4]{IL} and \cite[Lemma 7.3]{SSX}.
	\end{proof}

\noindent Based on Lemma \ref{lemb} and using the same arguments of \cite[Proposition 3.1]{LXX}, we have the following bounds in two scenarios, which ensures that $\E|S_{N}-T_{N}|^2=o(A_N^2)$.
\begin{proposition} \label{prop31} Suppose that assumptions $({\bf A1})$ and $({\bf A2})$ hold. 
	$($i$)$ If $1<\al\leq 3/2$ and $\beta=1$, then for any $\al'\in(1,\al)$, there exists a constant $c_{\al'}$ such that $
		\E|S_{N}-T_{N}|^2\leq c_{\al'}N^{4-2\al'}$.
 
	\noindent$($ii$)$ If $\al\be> 3/2$, then there exists a constant $c_{2}$ such that $
		\E|S_{N}-T_{N}|^2\leq c_{2}N$.
\end{proposition}

The weak convergences of $A_N^{-1}S_N$ and $A_N^{-1}T_N$ are thus equivalent.
We note that a more common form of $T_N$ in (\ref{T}) is as below, which separates the coefficients $a_j$ and innovations $\eps_n$,
\begin{equation}\label{T2}
	T_N=\Big(\sum_{n=1}^{N}\sum_{j=1}^{N}+ \sum_{n=-\infty}^{0}\sum_{j=1-n}^{N-n} - \sum_{n=1}^{N}\sum_{j=N-n+1}^{N} \Big)\big[K_\infty(a_j\eps_{n})-\E K_\infty(a_j\eps_{n})\big].
\end{equation}
Under the assumptions of Theorems~\ref{thm1}, \ref{thm2} or \ref{thm4}, we will see later that the primary contribution arises from the first double sum on the right-hand side of  (\ref{T2}).

{\noindent\bf Short memory region.}\\
Recall that $A_N\asymp\sqrt{N}$ (see Theorem~\ref{thm3}) in the short memory region.
Consequently, Proposition~\ref{prop31} does not provide a good $L_2$ approximation for $S_N$.
Instead, we approximate $X_n$s by truncation, making them locally dependent.
More precisely, as in \cite{Hsing,PT03}, for each $l\in\N$, we define $X_{n,l}:=\sum_{j=1}^{l}a_j\eps_{n-j}$ and  partial sums
\begin{equation}\label{Tl}
	S_{N,l}:=\sum_{n=1}^{N}\big[K(X_{n,l})-\E K(X_{n,l})\big].
\end{equation}

The proof strategy of \cite{Hsing,PT03} also works under the conditions of Theorem~\ref{thm3}.
Indeed, since $S_{N,l}$ is a sum of locally dependent and identically distributed random variables, it follows that as $N\to\infty$, $S_{N,l}/\sqrt{N}$ converges to a centered normal random variable with variance $\theta^2_l$.
Moreover, the central limit theorem for the rescaled sum $S_N/\sqrt{N}$ remains valid, if we show that as $l\to\infty$, (i) the $L_2$ difference $\E|S_N-S_{N,l}|^2$ is asymptotically $o(N)$, and (ii) the variance $\theta^2_l$ converges to a finite limit $\theta^2\in[0,\infty)$.
These points will be discussed in detail and established in Section~\ref{S:thm3}.

\section{Proof of Theorem \ref{thm1}} \label{S:thm1}

We start with the following lemma that compares the slowly varying functions related to $\ell$ and $L$.
\begin{lemma}  \label{lemslow} Recall $A_N$ from \eqref{sca1}. For any $\alpha\in (1,2)$, 
\[
\lim_{N\to\infty}\frac{\ell(A_N)}{L(N)-L(A_N)}=\lim_{N\to\infty}\frac{\ell(N)}{L(N)-L(A_N)}=0.
\]
\end{lemma}

\begin{proof}
	Observe that 
	\begin{align*}
		\liminf_{N\to\infty}\frac{L(N)-L(A_N)}{\ell(A_N)}
		&\geq\liminf_{N\to\infty}\int^N_{A_N+1}\frac{x^{-1}\ell(x)}{\ell(A_N)}dx\\
&\geq \int^{\infty}_1 \liminf_{N\to\infty} 1_{(1+A^{-1}_N,NA^{-1}_N)}(y)\frac{(A_N y)^{-1}\ell(A_Ny)}{A^{-1}_N\ell(A_N)}dy\geq \int^{M}_2 y^{-1}dy=\ln (M/2)
	\end{align*}
	for all $M\geq 2$. Letting $M\uparrow\infty$ gives $\lim_{N\to\infty}\frac{\ell(A_N)}{L(N)-L(A_N)}=0$. 
Similarly,
\begin{align*}
		\liminf_{N\to\infty}\frac{L(N)-L(A_N)}{\ell(N)}
&\geq \int^{1}_0 \liminf_{N\to\infty} 1_{(N^{-1},1)}(y)\frac{(N y)^{-1}\ell(Ny)}{N^{-1}\ell(N)} dy\geq \int^{1}_{1/M}y^{-1}dy=\ln M	
\end{align*}
	for all $M\geq 2$. Letting $M\uparrow\infty$ gives $\lim_{N\to\infty}\frac{\ell(N)}{L(N)-L(A_N)}=0$.
\end{proof} 

Recall from (\ref{T2}) that the approximation $T_N$ of $S_N$ can be rewritten as
\begin{align}\label{T3}
T_N=\text{I}_{N,1}+\text{I}_{N,2}-\text{I}_{N,3},
\end{align}
where we denote
\begin{align*}
\text{I}_{N,1}:=\sum^N_{n=1}\sum^{N}_{j=1} \xi_{n,j},\quad 
\text{I}_{N,2}:=\sum^0_{n=-\infty}\sum^{N-n}_{j=1-n} \xi_{n,j},\quad
\text{I}_{N,3}:=\sum^N_{n=1}\sum^{N}_{j=N-n+1} \xi_{n,j}
\end{align*}
with $\xi_{n,j}:=K_{\infty}(a_j\varepsilon_n)-\E K_{\infty}(a_j\varepsilon_n)$.
To show Theorem~\ref{thm1}, we will show in Proposition~\ref{prop3} that $A_N^{-1}\text{I}_{N,1}$ converges to a nontrivial limit, and show in Propositions~\ref{prop32}~and~\ref{prop34} that the other two terms are negligible after scaling. Let $\overset{\mathcal{L}}{\rightarrow}$ denote the convergence in law.

\begin{proposition} \label{prop3} 
	Under the assumptions of Theorem \ref{thm1}, as $N\to \infty$,
	\[
	A^{-1}_{N} \mathrm{I}_{N,1} \overset{\mathcal{L}}{\rightarrow} K'_\infty(0) Z^{\alpha}_{1},
	\]
	where $Z^{\alpha}_{1}$ is $\alpha$-stable with characteristic function given in \eqref{stablep}.
\end{proposition}
\begin{proof}
	The idea is to approximate $\xi_{n,j}$ by the linear term $K'_\infty(0)a_j\eps_n$ such that $\mathrm{I}_{N,1}$ is approximately a sum of i.i.d.\ random variables.
	However, a key observation is that the linear approximation is good only for $j\geq A_N$.
	Thereby, we rewrite
\[
\text{I}_{N,1}=\text{I}_{N,1,1}+\text{I}_{N,1,2}+\text{I}_{N,1,3} :=\sum^N_{n=1}\sum^{[N^{\delta}]}_{j=1} \xi_{n,j} +\sum^N_{n=1}\sum^{[A_N]}_{j=[N^{\delta}]+1} \xi_{n,j} +\sum^N_{n=1}\sum^{N}_{j=[A_N]+1} \xi_{n,j},
\]
where $\delta\in (0,\frac{2-\alpha}{2\alpha})$ is a constant.
We then proceed the linear approximation for the last term,
and show that the second moments of the first two terms are $o(A^2_N)$.

By assumptions ({\bf A1})-({\bf A2}), $K_{\infty}$ is bounded, and thus $\mathbb{E}|\text{I}_{N,1,1}|^2$ is bounded from above by a constant multiple of $N^{1+2\delta}$. 
This leads to $\lim_{N\to\infty} A_N^{-2}\hspace{1pt}\E|\text{I}_{N,1,1}|^2=0$.

For the term $\text{I}_{N,1,2}$, the Plancherel formula (e.g., see \cite{LXX}) yields that
\begin{align*}
\mathbb{E}|
\text{I}_{N,1,2}|^2
&=\frac{N}{4\pi^2}\int_{\mathbb{R}^2} \widehat{K}(u) \phi(-u)\widehat{K}(-v) \phi(v)\\
&\qquad\times \sum^{[A_N]}_{j_1,j_2=[N^{\delta}]+1}\mathbb{E}\Big[(e^{-\iota u a_{j_1} \varepsilon_{n}}-\phi_{\varepsilon}(-a_{j_1} u))(e^{\iota v a_{j_2} \varepsilon_{n}}-\phi_{\varepsilon}(a_{j_2} v))\Big]\, du\, dv\\
&\leq  N\Big(\sum^{[A_N]}_{j=[N^{\delta}]+1}\int_{\mathbb{R}} |\widehat{K}(u)||\phi(-u)|(1-|\phi_{\varepsilon}(-a_{j} u)|^2)^{\frac{1}{2}} du\Big)^2,
\end{align*}
where $\phi$ and $\phi_\eps$ are  characteristic functions of $X_1$ and $\eps_1$, respectively. Recalling (\ref{char}),
\begin{align*}
&\sum^{[A_N]}_{j=[N^{\delta}]+1}\int_{\mathbb{R}} |\widehat{K}(u)||\phi(-u)|(1-|\phi_{\varepsilon}(-a_{j} u)|^2)^{\frac{1}{2}} du\\
&\leq c_1 \sum^{[A_N]}_{j=[N^{\delta}]+1}\int_{|u|\leq N^{\frac{\delta}{2}}} |\widehat{K}(u)||\phi(-u)||a_{j} u|^{\frac{\alpha}{2}}\frac{H^{\frac{1}{2}}(|a_{j} u|^{-1})}{H^{\frac{1}{2}}(|a_{j} |^{-1})}H^{\frac{1}{2}}(|a_{j}|^{-1})\, du\\ 
&\qquad\qquad+c_1 A_N \int_{|u|>N^{\frac{\delta}{2}}} |\widehat{K}(u)||\phi(-u)|\, du\\
&\leq c_2\sum^{[A_N]}_{j=[N^{\delta}]+1}|a_{j} |^{\frac{\alpha}{2}}H^{\frac{1}{2}}(|a_{j}|^{-1})
\leq c_3 A_N^{1-\frac{\alpha}{2}}\ell^{\frac{\alpha}{2}}(A_N)H^{\frac{1}{2}}(A_N\ell^{-1}(A_N)),
\end{align*}
where in the second inequality we used the Potter bounds \cite[Theorem 1.5.6]{BGT89} that for any $\delta'>0$, there exists $c_{\delta'}$ such that $H(|a_ju|^{-1})/H(|a_j|^{-1})\leq c_{\delta'}|u|^{\delta'}$ for $|u|\leq N^{\delta'/2}$, and used the fact that $\phi$ decays faster than any polynomial as $|u|\to\infty$ so that the integrals are finite. 

Hence, by the above inequalities, the Potter bounds \cite[Theorem 1.5.6]{BGT89} and Remark \ref{rem1} (1),
\begin{align*}
\mathbb{E}|
\text{I}_{N,1,2}|^2
&\leq c_4 N A_N^{2-\alpha}\ell^{\alpha}(A_N)H(A_N\ell^{-1}(A_N))\\[5pt]
&=c_4 A^2_N \Big(\frac{\ell(A_N)}{L(N)-L(A_N)}\Big)^{\alpha}\frac{H(A_N\ell^{-1}(A_N))}{H(N_{\al, H_{\al}})}\frac{H(N_{\al, H_{\al}})}{H_{\alpha}(N)}\leq c_5 A^2_N \Big(\frac{\ell(A_N)}{L(N)-L(A_N)}\Big)^{\frac{\alpha}{2}}.
\end{align*}
Lemma~\ref{lemslow} then implies $\lim_{N\to\infty} A_N^{-2}\hspace{1pt}\E|\text{I}_{N,1,2}|^2=0$.

For $\text{I}_{N,1,3}$, we split it into the linear part $\text{I}_{N,1,3,2}$ and the remainder $\text{I}_{N,1,3,1}$,
\[
\text{I}_{N,1,3}=\text{I}_{N,1,3,1}+\text{I}_{N,1,3,2}:=\sum^N_{n=1}\sum^{N}_{j=[A_N]+1}\zeta_{n,j} + \sum^N_{n=1}\sum^{N}_{j=[A_N]+1} K'_{\infty}(0) a_j\varepsilon_n,
\]
where 
\begin{align}\label{zeta}
	\zeta_{n,j}:=K_{\infty}(a_j\varepsilon_n)-\E K_{\infty}(a_j\varepsilon_n)-K'_{\infty}(0)a_j\varepsilon_n.
\end{align}
By the Taylor expansion of $K_\infty$ and by the Plancherel formula, it is not hard to see (e.g., as in \cite[Section~3.2]{LXX}, or similar argument as for (\ref{Ksum}) below) that for $\alpha'\in(1,\alpha)$,
\begin{align*}
\mathbb{E}|\text{I}_{N,1,3,1}|
&\leq c_6 N\sum^N_{j=[A_N]+1} |a_j|^2 \mathbb{E}\big[|\varepsilon_n|^2 1_{\{|\varepsilon_n|\leq N_{\al, H_{\al}}\}}\big]
+c_6 N\sum^N_{j=[A_N]+1} |a_j|^{\alpha'} \mathbb{E}\big[|\varepsilon_n|^{\alpha'} 1_{\{|\varepsilon_n|>N_{\al, H_{\al}}\}}\big]  \\[5pt]
&\leq c_7 A^{-1}_N\ell^2(A_N) N  (N_{\al, H_{\al}})^{2-\alpha}H(N_{\al, H_{\al}}) 
+c_7 A_N^{1-\alpha'}\ell^{\alpha'}(A_N)N(N_{\al, H_{\al}})^{\alpha'-\alpha}H(N_{\al, H_{\al}})\\[5pt]
&\leq c_8 A_N\Big[\big(\frac{\ell(A_N)}{L(N)-L(A_N)}\big)^{2}+\big(\frac{\ell(A_N)}{L(N)-L(A_N)}\big)^{\alpha'}\Big].
\end{align*}
Thanks to Lemma~\ref{lemslow}, $\lim_{N\to\infty} A_N^{-1}\mathbb{E}|\text{I}_{N,1,3,1}|=0$. Since $
\text{I}_{N,1,3,2}=K'_\infty(0)\big(L(N)-L([A_N])\big)\sum^N_{n=1}\eps_n$, a direct computation using \eqref{char} yields that for any $u \in \mathbb{R}$,
\begin{align*}
	\lim_{N\to\infty}\ln \mathbb{E} \big[e^{\iota u A^{-1}_{N} C^{-1}_{K} \text{I}_{N,1,3,2}}\big]
	&=-\sigma|u|^{\alpha} (1- \iota D \sgn(u)),
\end{align*}
which is the characteristic function (\ref{stablep}) of $Z_1^\alpha$. This completes the proof.
\end{proof}

We then deal with $\text{I}_{N,2}$ and $\text{I}_{N,3}$. 
Recall $\zeta_{n,j}$ from (\ref{zeta}). Observe that 
\begin{align*}
	\text{I}_{N,2}&=\text{I}_{N,2,1}+\text{I}_{N,2,2} = \sum^0_{n=-\infty}\sum^{N-n}_{j=1-n} \zeta_{n,j}+ \sum^0_{n=-\infty}\sum^{N-n}_{j=1-n} K'_{\infty}(0) a_j\varepsilon_n, \\[5pt]
	\text{I}_{N,3}&=\text{I}_{N,3,1}+\text{I}_{N,3,2} = \sum^N_{n=1}\sum^{N}_{j=N-n+1} \zeta_{n,j}+ \sum^N_{n=1}\sum^{N}_{j=N-n+1} K'_{\infty}(0) a_j\varepsilon_n.
\end{align*}

\begin{proposition} \label{prop32} Suppose that assumptions {\rm({\bf A1})-({\bf A2})} hold, and $\alpha\in (1,2]$.  For any $\delta\in (1,\alpha)$, there exists a positive constant $c_{\delta,3}>0$ such that $
\mathbb{E} | \mathrm{I}_{N,2,1}|+\E |\mathrm{I}_{N,3,1} |\leq c_{\delta,3}\, N^{2-\delta}$.
\end{proposition}
\begin{proof} Note that by the Plancherel formula
\begin{align*}
\zeta_{n,j}
&=\frac{1}{2\pi}\int_{\R} \widehat{K}(u) \phi(-u)(e^{-\iota u a_j \varepsilon_{n}}-\phi_{\varepsilon}(-a_j u)+\iota u a_j \varepsilon_{n}) du.
\end{align*}
For any $\alpha'\in (\delta,\alpha)$, by assumptions ({\bf A1})-({\bf A2}) and Lemma \ref{lemb},
\begin{align} \label{Ksum}
|\zeta_{n,j}|   \nonumber
&\leq \int_{\R} |\widehat{K}(u)| |\phi(-u)|(|e^{-\iota u a_j \varepsilon_{n}}-1+\iota u a_j \varepsilon_{n}|+|1-\phi_{\varepsilon}(-a_j u)|) du\\ 
&\leq c_{1}\, \int_{\R} |\widehat{K}(u)| |\phi(u)|( |u a_j \varepsilon_{n}|^{\alpha'}+|a_j u|^{\alpha'}) du
\hspace{1pt}\leq c_{2}\, \big(|a_j|^{\alpha'} |\varepsilon_n|^{\alpha'}+|a_j|^{\alpha'}\big).
\end{align}
By \cite[Theorem 2.6.4 ]{IL}, $\mathbb{E} |\varepsilon_n|^{\alpha'}<\infty$ since $\al'<\al$. Then, for $\delta\in (1,\alpha')$, by the Potter bounds \cite[Theorem 1.5.6]{BGT89}, we have
\begin{align*}
	&\E |\text{I}_{N,2,1}|+\E |\text{I}_{N,3,1}| \leq c_{3}\,\sum^N_{n=1}\sum^{\infty}_{j=N-n+1} |a_j|^{\alpha'}+c_{3}\,\sum^0_{n=-\infty}\sum^{N-n}_{j=1-n} |a_j|^{\alpha'}\\
	&\quad\leq c_{4}\,\sum^N_{n=1}\sum^{\infty}_{j=N-n+1} j^{-\delta}+c_{4}\,\sum^{N}_{n=0}\sum^{N+n}_{j=1+n} j^{-\delta}+c_{4}\,\sum^{\infty}_{n=N+1}\sum^{N+n}_{j=1+n} j^{-\delta}\\
	&\quad\leq c_{5}\,\big(1+\sum^{N-1}_{n=1} (N-n)^{1-\delta}\big)+c_{5}\,\big(1+\sum^{N}_{n=1}  n^{1-\delta}\big)+c_{5}\,\sum^{\infty}_{n=N+1}\big(n^{1-\delta}-(N+n)^{1-\delta}\big)\\
	&\quad\leq c_{6}\,N^{2-\delta}+c_{6}\,N^{2-\delta}\int^{\infty}_1(x^{1-\delta}-(1+x)^{1-\delta})\, dx\leq c_{7}\,N^{2-\delta},
\end{align*}
where in the last inequality we used $1<\delta<2$ and $0<x^{1-\delta}-(1+x)^{1-\delta}\leq x^{-\delta}$ for all $x>0$.
\end{proof}

\begin{proposition} \label{prop34} Suppose that assumptions {\rm({\bf A1})-({\bf A2})} hold, $\alpha\in (1,2]$.  Then, as $N\to\infty$,
\[
A^{-1}_N (|\mathrm{I}_{N,2,2}|+|\mathrm{I}_{N,3,2}|) \overset{\mathbb{P}}{\longrightarrow} 0.
\]
\end{proposition}
\begin{proof} 
It suffices to consider the case when $K'_{\infty}(0)=1$. For any $n\geq 0$, let $a^N_n=\sum^{N+n}_{j=1+n}a_j$. Then, for any $u\in\mathbb{R}$ with $u\neq 0$, by (\ref{char}), as $N\to\infty$,
\begin{align*}
\ln \mathbb{E}\big[e^{\iota u A^{-1}_N \text{I}_{N,2,2}}\big]
&=-(\sigma+o(1))\sum^{\infty}_{n=0}|\frac{u a^N_n}{A_N}|^{\alpha}H(|\frac{u a^N_n}{A_N}|^{-1})(1-\iota D\text{sgn}(\frac{u a^N_n}{A_N})).
\end{align*}
So we only need to show $
\lim_{N\to\infty} \sum^{\infty}_{n=0}|\frac{a^N_n}{A_N}|^{\alpha}H(|\frac{a^N_n}{A_N}|^{-1})=0$.

For any $\eta\in(0,\alpha-1)$ and $\delta\in(0,1)$, by the Potter bounds \cite[Theorem 1.5.6]{BGT89} 
and Remark~\ref{rem1}~(1),
\begin{align*}
	&\limsup\limits_{N\to\infty} \sum^{\infty}_{n=0}|\frac{a^N_n}{A_N}|^{\alpha}H(|\frac{a^N_n}{A_N}|^{-1}) 
	\leq\limsup\limits_{N\to\infty} \frac{1}{N}\sum^{\infty}_{n=0}  \Big|\frac{a^N_n}{L(N)-L(A_N)}\Big|^{\alpha}\hspace{1pt}\frac{H\big(N^{\frac{1}{\alpha}}H^{\frac{1}{\alpha}}_{\alpha}(N)\frac{L(N)-L(A_N)}{a^N_n}\big)}{H_{\alpha}(N)}\\
	&\quad\leq  \limsup \limits_{N\to\infty} \bigg(\sum^{[A_N]}_{n=0}+\sum^{[\delta N]}_{n={[A_N]}+1}+\sum^{\infty}_{n=[\delta N]+1}\bigg)\frac{c_1}{N} \Big|\frac{a^N_n}{L(N)-L(A_N)}\Big|^{\alpha-\eta}.
\end{align*}
For the first two sums, the boundedness of summands yields an upper bound $2c_1\delta$.
For the last sum,
\begin{align*}
\limsup \limits_{N\to\infty}  \frac{c_1}{N}\sum^{\infty}_{n=[\delta N]+1}\Big|\frac{a^N_n}{L(N)-L(A_N)}\Big|^{\alpha-\eta}
\leq\limsup \limits_{N\to\infty} \frac{c_1}{N}\sum^{\infty}_{n=[\delta N]+1}\Big|\frac{Nn^{-1}\ell(n)}{L(N)-L(A_N)}\Big|^{\alpha-\eta}\\[5pt]
\quad\leq  \limsup \limits_{N\to\infty} \frac{c_2(\delta N+1)}{N}\Big|\frac{\frac{N}{\delta N+1}\ell(\delta N+1)}{L(N)-L(A_N)}\Big|^{\alpha-\eta}
=c_2\delta^{1+\eta-\alpha}\limsup \limits_{N\to\infty} \Big|\frac{\ell(N)}{L(N)-L(A_N)}\Big|^{\alpha-\eta}
=0,
\end{align*}
where we used Lemma \ref{lemslow} in the last inequality. Letting $\delta\downarrow 0$, we have $A^{-1}_N \text{I}_{N,2,2}\overset{\mathbb{P}}{\longrightarrow} 0$ as $N\to\infty$.

As for the term $\text{I}_{N,3,2}$, we only need to show
\begin{align*}
\lim\limits_{N\to\infty} \sum^{N}_{n=1}\Big|\frac{\sum^N_{j=N-n+1}a_j}{A_N}\Big|^{\alpha}H\Big(\Big|\frac{\sum^N_{j=N-n+1}a_j}{A_N}\Big|^{-1}\Big)=0.
\end{align*}
Similarly, for any $\eta\in (0,\alpha-1)$, we could obtain that for large $N$,
\begin{align*}
\sum^{N}_{n=1}\Big|\frac{\sum^N_{j=N-n+1}a_j}{A_N}\Big|^{\alpha}H\Big(\Big|\frac{\sum^N_{j=N-n+1}a_j}{A_N}\Big|^{-1}\Big)\leq c_3
N^{-1}\sum^{N}_{n=1}\Big|\frac{\sum^N_{j=N-n+1}a_j}{L(N)-L(A_N)}\Big|^{\alpha-\eta}.
\end{align*}
Note that $\sum^N_{n=1}\big|\sum^N_{j=N-n+1}a_j\big|^{\alpha-\eta}\leq \sum^{\infty}_{n=0} |a^N_n|^{\alpha-\eta}$.
The desired result then follows from similar arguments for the term $\text{I}_{N,2,2}$. \end{proof} 

\medskip
\noindent
{\bf Proof of Theorem \ref{thm1}}: 
The weak convergence of one-dimensional distributions immediately follows from  Propositions \ref{prop31}, \ref{prop32} and \ref{prop34}.

To show the convergence of finite-dimensional distributions, we only need to show that for any $t_2>t_1>0$, the increment $A_N^{-1}(S_{[Nt_2]}-S_{[Nt_1]})$ is asymptotically independent of $A_N^{-1}S_{[Nt_1]}$ and converges to $K'_\infty(0)Z^\al_{t_2-t_1}$.
Note that the only term that does not vanish in the limit is $\text{I}_{[Nt_1],1}$.
\begin{align*}
	\text{I}_{[N t_2],1} - \text{I}_{[N t_1], 1} &= \Big( \sum_{n=[N t_1] +1}^{[N t_2]} \sum_{j=1}^{[N t_2]} + \sum_{n=1}^{[N t_1]} \sum_{j=[N t_1]+1}^{[N t_2]} \Big) \xi_{n,j}.
\end{align*}
The first double sum is independent of $\text{I}_{[N t_1], 1}$ and converges weakly to $K'_\infty(0)Z^\al_{t_2-t_1}$ after scaling. It remains to bound the second double sum. Note that
\begin{equation*}
	\sum_{n=1}^{[N t_1]} \sum_{j=[N t_1]+1}^{[N t_2]} \xi_{n,j} = \sum_{n=1}^{[N t_1]} \sum_{j=[N t_1]+1}^{[N t_2]} \zeta_{n,j} + \sum_{n=1}^{[N t_1]} \sum_{j=[N t_1]+1}^{[N t_2]} K'_\infty(0)a_j\eps_n =:\text{II}_{N, t_1, t_2,1} + \text{II}_{N, t_1, t_2,2},
\end{equation*}
where $\zeta_{n,j}$ is defined in (\ref{zeta}).
The same argument as in Proposition~\ref{prop32} yields  $\mathbb{E} \big| \text{II}_{N, t_1, t_2,1} \big| \leq c_1 N^{2-\delta}$ for any $\delta\in(1,2)$. Observe that $\sum_{j=[N t_1]+1}^{[N t_2]} a_j \leq c_2\, \ell(N)$ by the uniform convergence theorem.  Now, using similar arguments as in Proposition \ref{prop34} and Lemma~\ref{lemslow}, we see that
$A_N^{-1} \text{II}_{N, t_1, t_2,2} \overset{\mathbb{P}}{\longrightarrow} 0$ as $N\to\infty$.
This completes the proof.

\section{Proof of Theorem \ref{thm2}} \label{S:thm2}

In the case $\alpha\beta=2$ with $\beta>1$ or $\alpha=2, \beta=1, K^{\prime}_{\infty}(0) = 0$, we will first obtain the following proposition which gives the main contribution in the partial sum $S_N$. Let
\begin{align} \label{staun}
\widetilde{T}_{N}=\sum^{N}_{n=1} \sum^{\infty}_{j=1} \big(K_{\infty}(a_j \varepsilon_{n})-\E K_{\infty}(a_j \varepsilon_{n}) \big).
\end{align}
For $\al=2,\beta=1$ and $K^{\prime}_{\infty}(0) = 0$, the same estimate as (\ref{Ksum}) gives that for $\al'\in(1,\al)$,
\begin{equation*}
	|K_{\infty}(a_j \varepsilon_{n})-\E K_{\infty}(a_j \varepsilon_{n})|=
	|K_{\infty}(a_j \varepsilon_{n})-\E K_{\infty}(a_j \varepsilon_{n})-K'_{\infty}(0)a_j \varepsilon_{n}| \leq
	c_1 (|a_j|^{\alpha'}|\varepsilon_n|^{\alpha'}+|a_j|^{\al'}).
\end{equation*}
For $\beta>1$, similarly by assumptions ({\bf A1})-({\bf A2}) and Lemma \ref{lemb}, 
\begin{align*}
&|K_{\infty}(a_j \varepsilon_{n})-\E K_{\infty}(a_j \varepsilon_{n})|
\leq \int_{\mathbb{R}}|\widehat{K}(u)||\phi(-u)|(|e^{-\iota u a_j\varepsilon_n}-1|+|1-\phi_{\varepsilon}(-a_ju)|)du\\
&\quad\leq c_2 \int_{\mathbb{R}}|\widehat{K}(u)||\phi(-u)|(|u a_j\varepsilon_n|^{\alpha'\wedge 1}+|a_ju|^{\alpha'})du
\leq c_3 (|a_j|^{\alpha'\wedge 1}|\varepsilon_n|^{\alpha'\wedge 1}+|a_j|^{\alpha'}).
\end{align*}
In both cases, the two upper bounds above are summable in $j$. Hence, the infinite series with respect to $j$ in (\ref{staun}) converges in both cases, and $\widetilde{T}_{N}$ is well defined. Using similar arguments as in the proofs of \cite[Lemmas 3.3 and 3.5]{LXX}, we have the following estimates for $T_{N}-\widetilde{T}_{N}$.

\begin{lemma} \label{lm51} Suppose that assumptions $(${\bf A1}$)$-$(${\bf A2}$)$ hold, $\alpha\beta=2$ and $\beta>1$ or $\alpha=2, \beta=1, K^{\prime}_{\infty}(0) = 0$.  For any $\alpha'\in (0,\alpha)$, $\beta'\in (1,\beta)$ with $\alpha\in (0,1]$ and $\alpha'\beta'>1$, there is a positive constant $c_{\alpha',\beta',3}>0$ such that
\[
\E |T_{N}-\widetilde{T}_{N}| \leq c_{\alpha',\beta',3}\, N^{2-\alpha'\beta'}.
\]
For any $\alpha'\in (1,\alpha)$, $\beta'\in (1,\beta)$ with $\alpha\in (1,2)$ and $\alpha'\beta'>1$, there is a positive constant $c_{\alpha',\beta',4}>0$ such that
\[
\E |T_{N}-\widetilde{T}_{N}|^{\alpha'} \leq c_{\alpha',\beta',4}\, N^{1+\alpha'(1-\beta')}.
\]
For any $\alpha'\in (1,2)$, $\beta'\in (0,1)$ with $\alpha=2, \beta=1, K^{\prime}_{\infty}(0) = 0$ and $\alpha'\beta'>1$, there is a positive constant $c_{\alpha',\beta',5}>0$ such that
\[
\E |T_{N}-\widetilde{T}_{N}| \leq c_{\alpha',\beta',5}\, N^{2-\alpha'\beta'}.
\]
\end{lemma}

The next Lemma reveals the regularity of slowly varying function $\ell$ at infinity under the assumption (\textbf{T1)} and will be used in the proof of Proposition \ref{prop52}.

\begin{lemma}\label{lem42}
	Under assumption {\rm(\textbf{T1)}}, $
	\lim_{x \rightarrow \infty} \ell(x \ell^{a}(x))/\ell(x) = 1$ locally uniformly in $a \in \mathbb{R}$. 
\end{lemma}

\begin{proof}
	By assumption {\rm(\textbf{T1)}} and the mean value theorem,
	\begin{align*}
&	\limsup_{x\to\infty}	\Big|\frac{\ell(2x)}{\ell(x)}-1\Big||\ln{\ell(x)}| 
=\limsup_{x\to\infty}	\Big|\frac{\ell(2x)-\ell(x)}{\ell(x)}\Big||\ln{\ell(x)}|
=\limsup_{x\to\infty}	x\Big|\frac{\ell'(x')}{\ell(x)}\Big||\ln{\ell(x)}| \\ 
&\quad\leq \limsup_{x\to\infty} |\eta(x')| \Big|\int_1^{x} \frac{\eta(t)}{t} dt + \ln{\sigma_0}\Big|\leq \limsup_{x\to\infty} \frac{c_1}{\ln^{\gamma}{x}} \big(\ln^{1-\gamma}{x} + 1\big)=0,
	\end{align*}
for some $x'\in (x,2x)$. 
Now the desired result immediately follows from \cite[Theorem 2.3.3]{BGT89}.
\end{proof}

For each $x\in\R$, let 
\begin{align} \label{etakx}
\eta_K(x)=\sum\limits^{\infty}_{j=1} \big(K_{\infty}(a_j x)-\E K_{\infty}(a_j \varepsilon_1)\big).
\end{align}
The convergence of the infinite series in (\ref{etakx}) follows from similar arguments as above. 
We note that by (\ref{staun}), $\widetilde{T}_{N}=\sum^{N}_{n=1}  \eta_K(\varepsilon_n)$ is a sum of i.i.d. random variables $\eta_K(\varepsilon_n)$. 
To obtain its limit distribution, we only need to derive the asymptotic behavior of $\widetilde{T}_{N}$.
\begin{proposition} \label{prop52} 
	Under the assumptions of Theorem \ref{thm2}, as $N\to\infty$,
\[
N^{-\frac{1}{2}}\overline{H}^{-\frac{1}{2}}_2(N) \widetilde{T}_{N} \overset{\mathcal{L}}{\longrightarrow} \sqrt{2(\gamma_1+\gamma_2)}\, W_1.
\]
\end{proposition}

\begin{proof}  
	This can be shown using the proof of \cite[Proposition 3.4]{LXX} with a minor adaption.
	Thus, we only sketch the proof and focus on the difference.

\noindent
{\bf Step 1} The same proof as for \cite[Proposition 3.4]{LXX} gives
\begin{align} \label{etak}
\lim\limits_{x\to \pm\infty} \frac{1}{|x|^{\frac{1}{\beta}} \ell^{\frac{1}{\beta}}(|x|^{\frac{1}{\beta}})} \eta_K(x)=\int^{\infty}_0 \big(K_{\infty}(\pm t^{-\beta})-K_{\infty}(0)\big)dt=C^{\pm}_K
\end{align}
except that an adaptation is needed for the term 
\[
\text{I}_3= \frac{1}{|x|^{\frac{1}{\beta}}\ell^{\frac{1}{\beta}}(|x|^{\frac{1}{\beta}})} \sum^{\infty}_{j=[M x^{\frac{1}{\beta}}\ell^{\frac{1}{\beta}}(x^{\frac{1}{\beta}})]+1} (K_{\infty}(a_j x)-\E K_{\infty}(a_j \varepsilon_1)), 
\] 
where $M$ is some large integer. In the case $\beta>1$, by the Lipschitz continuity of $K_{\infty}$ and \cite[Proposition 1.5.8]{BGT89},
\begin{align*} 
	&\limsup_{x\to\infty}|\text{I}_3|
	\leq c_1\limsup_{x\to\infty} \frac{1}{x^{\frac{1}{\beta}}\ell^{\frac{1}{\beta}}(x^{\frac{1}{\beta}})   } \sum^{\infty}_{j=[M x^{\frac{1}{\beta}}\ell^{\frac{1}{\beta}}(x^{\frac{1}{\beta}})]+1}  |a_j| x \nonumber\\[5pt]
	&\quad\leq c_2\, \limsup_{x\to\infty} x^{-\frac{1}{\beta}}\ell^{-\frac{1}{\beta}}(x^{\frac{1}{\beta}})    \big(M x^{\frac{1}{\beta}}\ell^{\frac{1}{\beta}}(x^{\frac{1}{\beta}})   \big)^{1-\beta}\ell\big(Mx^{\frac{1}{\beta}}\ell^{\frac{1}{\beta}}(x^{\frac{1}{\beta}})\big) x \nonumber
	=c_2\, M^{1-\beta}.
\end{align*} 
In the case $\beta=1, K^{\prime}_{\infty}(0) = 0$,
\begin{align*} 
	&\limsup_{x\to\infty}|\text{I}_3|
	=\limsup_{x\to\infty} \frac{1}{2\pi}\frac{1}{x \ell(x)} \bigg|\int_{\mathbb{R}} \widehat{K}(u)\phi(u) \sum^{\infty}_{j=[Mx \ell(x)]+1}  (e^{\iota  u a_j x}-1- \iota u a_j x)du \bigg| \nonumber\\
	&\quad\leq c_3\, \limsup_{x\to\infty} \frac{1}{x \ell(x)} \sum^{\infty}_{j=[M x\ell(x)]+1}  a^{2}_{j} x^{2} 
	\leq c_4\, \limsup_{x\to\infty} \frac{(M x \ell(x)   )^{-1}\ell^{2}(Mx \ell(x)) x^{2}}{x\ell(x)}
	=c_4\, M^{-1},  
\end{align*} 
where in the last equality we used $\lim_{x\to\infty} \ell(x\ell^{1/\beta}(x))/\ell(x)=1$.
Moreover, we also comment that for $\beta>1$, $C^{\pm}_K$ in (\ref{etak}) is finite, and that for $\beta=1, K^{\prime}_{\infty}(0) = 0$,
\begin{align*}
	\int^{\infty}_{0}\big(K_{\infty}(\pm t^{-1})-K_{\infty}(0)\big)dt
	&= \frac{1}{2 \pi} \int^{\infty}_{0} \int_{\R} \widehat{K}(u) \phi(u)(e^{\pm \iota u t^{-1}} -1) du\\
	&= \frac{1}{2 \pi} \int^{\infty}_{0} \int_{\R} \widehat{K}(u) \phi(u)(e^{\pm \iota u t^{-1}} -1 \mp \iota u t^{-1}) du < \infty.
\end{align*}

\noindent
{\bf Step 2} We next show that $\eta_K(\varepsilon_1)$ is in the domain of attraction of a normal distribution. That is,
\begin{align} \label{attraction}
\lim\limits_{x\to\infty} \frac{(R^{\leftarrow}_{\beta}(x))^{\alpha}}{h(R^{\leftarrow}_{\beta}(x))}\P\big(\eta_K(\varepsilon_1)>x\big)=\gamma_2\quad\text{and}\quad \lim\limits_{x\to\infty} \frac{(R^{\leftarrow}_{\beta}(x))^{\alpha}}{h(R^{\leftarrow}_{\beta}(x))}\P\big(\eta_K(\varepsilon_1)<-x\big)=\gamma_1,
\end{align}
where $R_{\beta, \ell, \beta}(x) = x^{\frac{1}{\beta}} \ell^{\frac{1}{\beta}}(x^{\frac{1}{\beta}})$ and $R^{\leftarrow}_{\beta}(x) = R^{\leftarrow}_{\beta, \ell, \beta}(x)$ are defined in \eqref{RInv}, 
\[
\gamma_2=\sigma_2 (C^+_K)^{2} 1_{\{C^+_K>0\}}+\sigma_1 (C^-_K)^{2} 1_{\{C^-_K>0\}}\;\; \text{and}\;\; \gamma_1=\sigma_2 |C^+_K|^{2} 1_{\{C^+_K<0\}}+\sigma_1 |C^-_K|^{2} 1_{\{C^-_K<0\}}.
\] 
\cite[Proposition 1.5.7 and Theorem 1.5.12]{BGT89} imply that $\frac{(R^{\leftarrow}_{\beta}(x))^{\alpha}}{h(R^{\leftarrow}_{\beta}(x))}$ is a regularly varying function at $\infty$ with index $2$. By assumption ({\bf T1}), there exists $A>0$ such that the regularly varying function $R_{\beta, \ell, \beta}(x)$ is continuous and strictly increasing on the interval $(A,\infty)$. So, with the same proof as Step 2 in \cite[Proposition 3.4]{LXX}, we could easily get \eqref{attraction}.

\noindent
{\bf Step 3}  By (\ref{attraction}), we could easily get that as $u\to 0$,
\begin{align*}
\ln\mathbb{E}[e^{\iota u \eta_k(\varepsilon_1)}]=-(\gamma_1+\gamma_2+o(1))|u|^2\overline{H}(|u|^{-1}),
\end{align*}
where $\overline{H}(t)=-\int^t_0 x^2 d \frac{h(R^{\leftarrow}_{\beta}(x))}{(R^{\leftarrow}_{\beta}(x))^{\alpha}}$. Therefore, for any $u\in\mathbb{R}$, 
\begin{align*}
\ln \mathbb{E}[e^{\iota u N^{-\frac{1}{2}}\overline{H}^{-\frac{1}{2}}_2(N) \widetilde{T}_{N}}]
&=-(\gamma_1+\gamma_2+o(1))|u|^2 \overline{H}(|u N^{-\frac{1}{2}}\overline{H}^{-\frac{1}{2}}_2(N)|^{-1})\overline{H}^{-1}_2(N)\\
&=-(\gamma_1+\gamma_2+o(1))|u|^2 
\end{align*}
as $N\to\infty$. This completes the proof.
\end{proof}

\medskip
\noindent
{\bf Proof of Theorem \ref{thm2}}:  Recall that $\widetilde{T}_{N}$ is a finite sum of i.i.d. random variables in the domain of attraction of the normal distribution. The process $\{\widetilde{T}_{[Nt]} : t \geq 0\}$ thus has independent increments. Using characteristic functions and Proposition \ref{prop52}, we can obtain that as $N\to\infty$,
	\begin{align}\label{F-T}
	\big\{N^{-\frac{1}{2}}\overline{H}^{-\frac{1}{2}}_2(N) \widetilde{T}_{[Nt]}: t\geq 0\big\} \overset{\rm f.d.d.}{\longrightarrow} \big\{ \ga W_t:\; t\geq 0 \big\}.
	\end{align}

    For any $t>0$, by Proposition \ref{prop31} and Lemma \ref{lm51}, $
	N^{-\frac{1}{2}}\overline{H}^{-\frac{1}{2}}_2(N) (S_{[Nt]} - \widetilde{T}_{[Nt]})$
	converges to 0 in $L_1$ norm as $N\to\infty$. Combining with \eqref{F-T}, the convergence of finite-dimensional distributions in Theorem \ref{thm2} follows.

\section{Proof of Theorem \ref{thm4}} \label{S:thm4}
In the case $\alpha=2, \beta=1, K^{\prime}_{\infty}(0) \neq 0$, by Proposition \ref{prop31} we see that the partial sum $S_N$ is well approximated by $T_{N}$. In this subsection, we will apply the Lindeberg-Feller CLT for $T_N$ to show Theorem \ref{thm4}. We start with the following lemma studying the slowly varying function $L$ at $\infty$ under the assumptions (\textbf{T1}) and (\textbf{T2}). It will be used in the proof of Proposition \ref{prop38}.
\begin{lemma}\label{lem51}
	Under assumptions {\rm(\textbf{T1})} and {\rm(\textbf{T2})}, $\lim_{x \rightarrow \infty} \frac{L(x L(x))}{L(x)} = 1$ and $\lim_{x \rightarrow \infty} L(x) = \infty$.
\end{lemma}

\begin{proof}
	Since $j^{-1}\ell(j)$ is the discrete derivative of $L(j)$,
	it is not hard to see that there exists a constant $C$ such that $L(x)\sim\widetilde{L}(x):=C+\int_{1}^{x} \frac{\ell(t)}{t} dt$.
	Without loss of generality, we write $L(x) = C+ \int_{1}^{x} \frac{\ell(t)}{t} dt$.
	Using $(\ln L(x))' = \frac{\ell(x)}{xL(x)}$, we have $L(x) = C e^{\int_{1}^{x} \frac{\ell(t)}{t L(t)} dt}$.

   Under assumption (\textbf{T1}), for $x>1$,
   \begin{align*}
   	\ell(x) \ln{x}  &= \int_{1}^{x} \frac{\ell(t)}{t} dt + \int_{1}^{x} \log{t}\hspace{1pt} d\ell(t) = \int_{1}^{x} \frac{\ell(t)}{t} dt + \int_{1}^{x} \frac{\ell(t)}{t} \eta(t) \ln{t} \hspace{1pt} dt \\
   	&\leq \begin{cases}
   	L(x) + L(x) (\ln{x})^{1-\gamma} & \text{ if } \eta(t) \leq (\ln{t})^{-\gamma} \quad \text{with $1/2 < \gamma \leq 1$} \vspace{3pt} \\ 
   	L(x)(1+o(1))	& \text{ if } \eta(t) \leq(\ln{t})^{-\gamma} \quad \text{with $\gamma > 1$.}
   	\end{cases}
   \end{align*}
Hence, by the proof of Lemma \ref{lem42}, $\lim_{x \rightarrow \infty} L(x L(x))/L(x)= 1$.

Recall that $L(1)=a_1>0$ and $L(x)$ is increasing. By assumption (\textbf{T2}), it is easy to see from iteration that $\lim_{x \rightarrow \infty} L(x) = \infty$.
\end{proof}

We now show the weak convergence of $T_N$.
\begin{proposition} \label{prop38} Under the assumptions of Theorem \ref{thm4}, as $N\to\infty$,
	\begin{align*}
	N^{-\frac{1}{2}} H_{2}^{-\frac{1}{2}}(N) L^{-1}(N)\, T_{N} \overset{\mathcal{L}}{\rightarrow}  \sqrt{c_{L, h} (\si_1+\si_2)}\, |K_{\infty}^{\prime}(0)| \, W_1
	\end{align*}
	where $c_{L, h}$, $\si_1$ and $\si_2$ are the same as those in Theorem~\ref{thm4}.
\end{proposition}
\begin{proof}
Recall from (\ref{T3}) that $\xi_{n,j}=K_\infty(a_j\eps_n)-\E K_\infty(a_j\eps_n)$ and
\begin{align*}
T_{N} = \text{I}_{N,1} + \text{I}_{N,2} - \text{I}_{N,3}
=\sum^N_{n=1}\sum^{N}_{j=1} \xi_{n,j}+\sum^0_{n=-\infty}\sum^{N-n}_{j=1-n} \xi_{n,j} -\sum^N_{n=1}\sum^{N}_{j=N-n+1} \xi_{n,j}.
\end{align*}
Applyiping the same proofs for Propositions~\ref{prop32}~and~\ref{prop34}, we can obtain that $A_N^{-1}(\text{I}_{N,2}-\text{I}_{N,3})\overset{\P}{\longrightarrow}0$ as $N\to\infty$, where $A_N=N^{\frac{1}{2}} H_{2}^{\frac{1}{2}}(N) L(N)$. It thus suffices to consider the limit of $A_N^{-1}\text{I}_{N,1}$.

Note that the Taylor expansion as used in (\ref{Ksum}) implies that for any $\al'\in(1,2)$,
\begin{align} \label{lipschitz1}
	|K_{\infty}(a_j x)-\E K_{\infty}(a_j \varepsilon_1) - K_{\infty}^{\prime}(0)a_{j}x |\leq c_1\, (|a_j x|^{\alpha'}+|a_j|^{\alpha'} ),
\end{align}
and the function in the left-hand side is Lipschitz in $x$.
As the above upper bound is summable in $j$, with $\zeta_{n,j}$ in (\ref{zeta}), we can rewrite
\begin{align*} 
	\text{I}_{N,1} =\text{I}_{N,1,1} - \text{I}_{N,1,2} + \text{I}_{N,1,3}:=
	\sum_{n=1}^{N} \sum_{j=1}^{\infty}  \zeta_{n,j} 
	- \sum_{n=1}^{N}  \sum_{j=N+1}^{\infty} \zeta_{n,j} 
	+ \sum_{n=1}^{N} \sum_{n=1}^{N}  K_{\infty}^{\prime}(0)a_{j}\varepsilon_{n}.
\end{align*}
Since the same proof as in Proposition \ref{prop32} implies that 
$\mathbb{E}|\text{I}_{N,1,2}| \leq c_2 N^{2-\delta}$ for $\delta \in (1,2)$,
\begin{align*}
	\lim_{N\to\infty} A_N^{-1}\mathbb{E}|\text{I}_{N,1,2}| =\lim_{N\to\infty} N^{-\frac{1}{2}} H_{2}^{-\frac{1}{2}}(N) L^{-1}(N) \mathbb{E}|\text{I}_{N,1,2}|=0.
\end{align*}
Observe that $\text{I}_{N,1,3}=\sum_{n=1}^{N}K'_\infty(0)L(N)\eps_n$. Denoting
\begin{align} \label{etak1x}
	\tilde\eta_K(x):=\sum\limits^{\infty}_{j=1}\zeta_{n,j}=\sum\limits^{\infty}_{j=1} \big(K_{\infty}(a_j x)-\E K_{\infty}(a_j \varepsilon_1) - K_{\infty}^{\prime}(0)a_{j}x\big),
\end{align}
we can write $\text{I}_{N,1,1}+\text{I}_{N,1,3}=\sum_{n=1}^{N}\big( \tilde\eta(\eps_n)+ K'_\infty(0)L(N)\eps_n \big)$.
The convergence to a normal distribution of such triangular array can then be obtained by the Lindeberg-Feller CLT, once we know the asymptotic behavior of the function $\tilde\eta_K$.

\noindent
\textbf{Step 1} 
We first show that
\begin{equation}\label{etak1}
	\lim_{x \rightarrow \pm \infty} \frac{1}{|x| L(|x|)} \tilde\eta_{K}(x) = \mp K^{\prime}_{\infty}(0). 
\end{equation}
By symmetry, we only consider $x>0$.
Fix a large constant $M>0$.
Since $K_\infty$ is bounded,
\begin{align}\label{tieta1}
	\limsup_{x \to \infty}\frac{1}{x L(x)} \sum_{j=1}^{[x L(x) /M]}  \big|K_{\infty}\left(a_{j} x\right)-\mathbb{E} K_{\infty}\left(a_{j} x\right)\big| \leq 	c_3 \limsup_{x \to \infty}\frac{x L(x) /M}{x L(x)} = \frac{c_3}{M}.
\end{align}
By Lemma~\ref{lem51}, we have $\lim_{x \rightarrow \infty} L(x L(x))/L(x)=1$ and thereby
\begin{align}\label{tieta2}
	\lim_{x \to \infty}\frac{1}{x L(x)} \sum_{j=1}^{[x L(x) /M]} -K_{\infty}^{\prime}(0)a_{j}x = -K_{\infty}^{\prime}(0)\lim_{x \to \infty}  \frac{L\big([x L(x)/M]\big)}{L(x)} = -K_{\infty}^{\prime}(0).
\end{align}
For the remaining sum, the Taylor expansion and \cite[Proposition 1.5.8]{BGT89} yield
\begin{align}
	&\limsup_{x \rightarrow \infty} \frac{1}{x L(x)}\sum_{j=[x L(x) /M]+1}^{\infty} \big| K_{\infty}\left(a_{j} x\right)- K_{\infty}(0) - K_{\infty}^{\prime}(0)a_{j}x \big| \notag \\ 
	&\quad\leq c_4 \limsup_{x \rightarrow \infty} \frac{1}{x L(x)} \sum_{j=[x L(x) /M]}^{\infty} |a_{j} x|^2 \notag
	\leq c_5\limsup_{x \rightarrow \infty} \frac{x^2}{x L(x)} \frac{\ell^{2}(x L(x) /M)}{x L(x) /M}   \notag \vspace{5pt} \\
	&\quad =  c_5 \limsup_{x \rightarrow \infty}  \hspace{1pt} M\hspace{1pt} \frac{\ell^{2}(x L(x))}{L^{2}(x)} =0, \label{Formula1.2}
\end{align}
where we used $\lim_{x \rightarrow \infty} L(x L(x))/L(x)=1$ and $\lim_{x \rightarrow \infty} \ell(x)/L(x)=0$ in the last equality.

Lastly, for any $\al'\in(1,2)$, using Plancherel formula, (\ref{lipschitz1}) and $\sum_{j=1}^\infty |a_j|^{\al'}<\infty$,
\begin{align}\label{tieta4}
	 \limsup_{x \rightarrow \infty}\frac{1}{x L(x)} \sum_{j=[x L(x) /M]+1}^{\infty} \big|K_{\infty}(0)-\mathbb{E} K_{\infty}\left(a_{j} \varepsilon_{n}\right)\big| \leq c_6 \limsup_{x \rightarrow \infty}\frac{1}{x L(x)}\sum\limits_{j=1}^\infty  |a_j|^{\al'}=0.
\end{align}

Then, (\ref{etak1}) follows by combining (\ref{tieta1})-(\ref{tieta4}) and sending $M\to\infty$.

\noindent
\textbf{Step 2} To obtain the weak limit of $A_N^{-1}(\text{I}_{N,1,1} + \text{I}_{N,1,3})$, we start with the truncation of $\eps_n$ at $N_{2,H_2}=N^{\frac{1}{2}} H_{2}^{\frac{1}{2}}(N)$.
Letting $\hat\eps_n:=\eps_n1_{\{|\eps_n|\leq N_{2,H_2}\}}$, by (\ref{h}), we have
\begin{align*}
	\lim_{N \rightarrow \infty} \sum_{n=1}^{N} \mathbb{P}(\eps_n\neq \hat\eps_n) 
	=\lim_{N \rightarrow \infty} N \mathbb{P}(|\eps_1|>N_{2,H_2}) 
	= (\sigma_1+\sigma_2) \lim_{N \rightarrow \infty} \frac{h(N^{\frac{1}{2}} H_{2}^{\frac{1}{2}}(N))}{H(N^{\frac{1}{2}} H_{2}^{\frac{1}{2}}(N))}
	= 0.
\end{align*}

We then check the two conditions of the Lindeberg-Feller CLT for
\begin{align}\label{trun}
A_N^{-1}\sum\limits_{n=1}^{N}  \big( \tilde\eta_{K}(\hat\varepsilon_{n}) + K^{\prime}_{\infty}(0) L(N) \hat{\varepsilon}_{n} \big).
\end{align}
We evaluate its variance by estimating the first and second moments.
For the first moment, by the definition (\ref{etak1x}) of $\tilde\eta_K$, we can rewrite $\tilde\eta_{K}(\hat\varepsilon_{n}) + K^{\prime}_{\infty}(0) L(N) \hat{\varepsilon}_{n}$ as
\begin{align}\label{1mt}
	\sum_{j=1}^{N} \big(K_\infty(a_j\hat\eps_n)-\E K_\infty(a_j\eps_n) \big) + \sum_{j=N+1}^{\infty} \big(K_\infty(a_j\hat\eps_n)-\E K_\infty(a_j\eps_n) -K'_\infty(0)L(N)\hat\eps_n \big).
\end{align}
Hence, by the Lipschitz continuity of $K_\infty$, (\ref{lipschitz1}) and (\ref{h}),
\begin{align*}
\lim_{N\to\infty}A_N^{-1} \Big|\mathbb{E}\big[\sum\limits_{n=1}^{N}  \big( \tilde\eta_{K}(\hat\varepsilon_{n}) + K^{\prime}_{\infty}(0) L(N) \hat{\varepsilon}_{n} \big)\big]\Big|
&\leq c_6 \lim_{N\to\infty} A_N^{-1} N  \big[\sum_{j=1}^{N} |a_j| \mathbb{E}[|\hat\eps_n-\eps_n|+\sum_{j=N+1}^{\infty} |a_j|^{\alpha'}\big]\\
&\leq c_7 \lim_{N\to\infty} N^{-1}_{2, H_2}h(N_{2,H_2})=0,
\end{align*}
where in the last equality we used $H_2(N)\sim H(N_{2,H_2})$ and $\lim\limits_{x\to\infty}h(x)/H(x)=0$.
Thus the first moment of (\ref{trun}) converges to 0. For the second moment of (\ref{trun}), we will estimate in each piece $|\hat\eps_n|\in I_N^{y,\upsilon}:=(N^y,N^{y+\upsilon}]$ with $\upsilon$ small, and then apply the Riemann sum approximation.
Denote $D_N^{y,\upsilon}:=\{|\hat\eps_1|\in I_N^{y,\upsilon}\}$.
For any $y\in(0,1/2)$, by the layer cake representation and (\ref{h}), as $N\to\infty$,
\begin{align*}
	&\E\big[\hat\eps_1^{\hspace{1pt}2}\hspace{1pt} 1_{D_N^{y,\upsilon}}\big]
	=2\int_0^\infty x \P\big(|\hat\eps_1| 1_{\{\hat\eps_1\in I_N^{y,\upsilon}\}}>x\big) dx \notag\\ 
	&\quad=(\sigma_1+\sigma_2+o(1))\Big(2 \int_{N^y}^{N^{y+\upsilon}} \frac{h(x)}{x}dx + h(N^{y}) -h(N^{y+\upsilon}) \Big).
\end{align*}
By a change of variable, the bounded convergence theorem and assumption ({\bf T2}),
\begin{align*}
	\lim_{N\to\infty}\frac{1}{h(N)\ln N}\int_{N^y}^{N^{y+\upsilon}} \frac{h(x)}{x}dx
	= \int_{y}^{y+\upsilon}\lim_{N\to\infty} \frac{h(N^x)}{h(N)}dx 
	= \int_{y}^{y+\upsilon} g_h(x)dx.
\end{align*}
Thus, we have
\begin{align} \label{2mt1}
	\lim_{N\to\infty}\frac{1}{h(N)\ln N} \E\big[\hat\eps^{\hspace{1pt}2}_1\hspace{1pt} 1_{D_N^{y,\upsilon}}\big] = 2(\si_1+\si_2)\int_{y}^{y+\upsilon} g_h(x)dx.
\end{align}
Concatenating the intervals $I_N^{y,\upsilon}$ to the whole domain, we see that as $N\to\infty$, 
\begin{align} \label{2mt4}
	 H_2(N)\sim H(N^{\frac{1}{2}}H^{\frac{1}{2}}_2(N)) \sim(\si_1+\si_2)^{-1}\E[\hat\eps_1^{\hspace{1pt}2}] 
	 \sim 2h(N)\ln N\int_0^{\frac{1}{2}}g_h(x)dx   .
\end{align}
We then use (\ref{etak1}) to replace $\tilde\eta_K(\hat\eps_1)$ by $-K'_\infty(0)L(N^y)\hat\eps_1$ on $D_N^{y,\upsilon}$ for small $\upsilon>0$.
Precisely, fixing $\upsilon_0\in (0,1/2)$ and taking $(y,y+\upsilon)\in[\upsilon_0,1/2-\upsilon_0]$, similar calculation as for (\ref{2mt1}) yields
\begin{align}\label{2mt2}
	\lim_{N\to\infty}\frac{\E\big[\big( L(N)\hat\eps_1 - L(N^y)\hat\eps_1 \big)^2 1_{D_N^{y,\upsilon}} \big]}{L^2(N) h(N)\ln N} 
	= 2(\si_1+\si_2)\int_{y}^{y+\upsilon} (1-g_L(y))^2 g_h(x)dx,
\end{align}
where we used $\lim_{N\to\infty} L(N^y)/L(N)=g_L(y)$ by ({\bf T2}).
To bound the remainder, we note that for any $\delta>0$ and $x,y\in[\upsilon_0,1/2-\upsilon_0]$, by uniform continuity there exists $\upsilon>0$ such that $|g_L(y)-g_L(x)|\leq\delta$ as long as $|x-y|<\upsilon$.
Therefore, for $|u|\in I_N^{y,y+\upsilon}$ and $x$ satisfying $N^x=|u|$,
\begin{align*}
	&\lim_{N\to\infty} \Big| \frac{\big(\tilde\eta_K(u)+K'_\infty(0)L(N)u\big) - K'_\infty(0)\big(- L(N^y)u+L(N)u\big) }{L(N) u} \Big| \\
	&\quad= |K'_\infty(0)|\lim_{N\to\infty}  \frac{|L(u) (1+o(1))-L(N^y)|}{L(N)}  = |K'_\infty(0)||g_L(y)-g_L(x)|\leq |K'_\infty(0)|\delta.
\end{align*}
Using (\ref{etak1}) and the monotonicity of $L$, we can similarly obtain
\begin{align*}
	&\limsup_{N\to\infty} \Big|  \frac{\big(\tilde\eta_K(u)+K'_\infty(0)L(N)u\big) + K'_\infty(0)\big(- L(N^y)u+L(N)u\big) }{L(N) u} \Big|\leq c_8.
\end{align*}
The above two bounds and (\ref{2mt1}) give 
\begin{align}\label{2mt3}
	&\limsup_{N\to\infty}\Big|\frac{\E\big[\big( \tilde\eta_K(\hat\eps_1) + K'_\infty(0)L(N)\hat\eps_1 \big)^2 1_{D_N^{y,\upsilon}} \big]}{L^2(N) h(N)\ln N} - 
	 \frac{\E\big[|K^{'}_\infty(0)|^2\big( L(N)\hat\eps_1 - L(N^y)\hat\eps_1 \big)^2 1_{D_N^{y,\upsilon}} \big]}{L^2(N) h(N)\ln N}\Big| \notag\\ 
	&\hspace{6.6cm} \leq c_9\, \delta \limsup_{N\to\infty} \frac{\E\big[\hat\eps_1^{\hspace{1pt}2} 1_{D_N^{y,\upsilon}}\big]}{h(N)\ln N} 
	\leq c_{10}\, \delta \int_y^{y+\upsilon} g_h(x)dx.
\end{align}
Note that $\delta$ can be made arbitrary small as $\upsilon\downarrow0$.
Using (\ref{2mt2}) and (\ref{2mt3}), and concatenating the intervals to $I_N^{\upsilon_0}:=[N^{\upsilon_0},N^{\frac{1}{2}-\upsilon_0}]$, the standard Riemann sum approximation implies
\begin{align}\label{2mt5}
	\lim_{N\to\infty}\frac{\E\big[\big| \tilde{\eta}(\hat{\eps}_1) + K'_\infty(0)L(N)\hat \eps_1 \big|^2 1_{\{\hat\eps_1\in I_N^{\upsilon_0}\}} \big]}{L^2(N) h(N)\ln N} 
	= (\si_1+\si_2)\int_{\upsilon_0}^{\frac{1}{2}-\upsilon_0} (1-g_L(x))^2 g_h(x)dx.
\end{align}
In the remaining domain, if $|\eps_1|<N^{\upsilon_0}$, then again by (\ref{etak1}), 
\begin{align}\label{2mt6}
	\limsup_{N\to\infty}\frac{\E\big[\big| \tilde{\eta}(\hat{\eps}_1) + K'_\infty(0) L(N)\hat\eps_1 \big|^2 1_{\{|\hat\eps_1|< N^{\upsilon_0}\}} \big]}{L^2(N) h(N)\ln N} 
	\leq  c_{11}\int_0^{\upsilon_0}  g_h(x)dx.
\end{align}
For $N^{1/2-\upsilon_0}<|\hat\eps_1|$, since $N_{2,H_2}\leq N^{1/2+\upsilon_0}$ for large $N$ and $|\hat\eps_1|\leq N_{2,H_2}$,
\begin{align}\label{lind}
	\limsup_{N\to\infty}\frac{\E\big[\big| \tilde{\eta}(\hat{\eps}_1) + K'_\infty(0) L(N)\hat\eps_1 \big|^2 1_{\{N^{1/2-\upsilon_0}< |\hat\eps_1|\}} \big]}{L^2(N) h(N)\ln N} 
	\leq  c_{12}\int_{\frac{1}{2}-\upsilon_0}^{\frac{1}{2}+\upsilon_0} (1-g_L(x))^2 g_h(x)dx.
\end{align}
Combining (\ref{2mt4}), (\ref{2mt5}), (\ref{2mt6}) and (\ref{lind}), and sending $\upsilon_0\downarrow0$, we obtain the variance
\begin{align*}
	&\lim_{N\to\infty}\frac{N}{A^{2}_{N}}\mathbb{E} \big( \tilde\eta_{K}(\hat\varepsilon_{1}) + K^{\prime}_{\infty}(0) L(N) \hat\varepsilon_{1} \big)^{2}
	= \lim_{N\to\infty}\frac{\E\big( \tilde{\eta}(\hat{\eps}_1) + K'_\infty(0)L(N)\hat\eps_1 \big)^2  }{L^2(N) H_2(N)} \notag \\ 
	&\hspace{3cm} = (\si_1+\si_2)|K^{'}_\infty(0)|^2 \frac{ \int_{0}^{1/2} (1-g_L(x))^2 g_h(x)dx}{ \int_0^{1/2}g_h(x)dx } = (\si_1+\si_2)|K^{'}_\infty(0)|^2 c_{L,h}.
\end{align*}

It remains to verify Lindeberg's condition. For any $\delta>0$ and $\upsilon_0\in (0,1/2)$, by (\ref{etak1}), we have
\begin{align*}
	D_\delta:=\{ | \tilde\eta_{K}(\hat\varepsilon_{1}) + K^{\prime}_{\infty}(0) L(N) \hat\varepsilon_{1}| \geq \delta A_N \} \subset 
	\{|\hat\eps_1|> N^{\frac{1}{2}-\upsilon_0}\}
\end{align*}
when $N$ is large. Thus, by (\ref{lind}) and sending $\upsilon_0\downarrow0$ therein, we have
\begin{align*}
	\lim_{N\to\infty}\frac{N}{A^{2}_{N}}\mathbb{E}\big[ \big( \tilde\eta_{K}(\hat\varepsilon_{1}) + K^{\prime}_{\infty}(0) L(N) \hat\varepsilon_{1}\big)^{2} 1_{D_{\delta}} \big] =0.
\end{align*}
The desired central limit theorem then follows.
\end{proof}


\medskip
\noindent
{\bf Proof of Theorem \ref{thm4}}:  
	(i) For $\al=2$, by the definition of $H$ in (\ref{H}) and Lemma \ref{lem51},
	\[
	\sum_{j=1}^{\infty}a_jH^{\frac{1}{2}}(a_j^{-1}) = \sum_{j=1}^{\infty}j^{-1} \ell(j) H^{\frac{1}{2}}(j \ell^{-1}(j)) \geq c_1\sum_{j=1}^{\infty}j^{-1} \ell(j) = c_1\lim_{x \rightarrow \infty} L(x) = \infty.
	\]
	Hence the linear process exhibits long memory.\\
	(ii) The one-dimensional convergence follows from Propositions~\ref{prop31} and \ref{prop38}.
	To show the convergence in finite-dimensional distributions, similar as before, it suffices to show that for $t_2 > t_1 > 0$, $A_N^{-1} (S_{[Nt_2]}-S_{[Nt_1]})$ and $A_N^{-1} S_{[Nt_1]}$ are asymptotically independent.
	In particular, we only need to show asymptotic increment independence of the process $\{A_N^{-1} \big(\text{I}_{[Nt],1,1}+\text{I}_{[Nt],1,3}\big)\}_{t\geq0}$. That is,
\begin{equation*}\begin{array}{c}
	\displaystyle A_N^{-1} \Big(\sum_{n=1}^{[Nt_1]} \tilde\eta_{K}(\varepsilon_{n})
	+   \sum_{n=1}^{[Nt_1]} K^{\prime}_{\infty}(0) L([Nt_1]) {\varepsilon}_{n} \Big) \quad\text{and} \vspace{5pt}\\ 
	\displaystyle A_N^{-1} \Big(\sum_{n=[Nt_1]+1}^{[Nt_2]} \hspace{-8pt} \tilde\eta_{K}(\varepsilon_{n})
	+   \hspace{-8pt} \sum_{n=[Nt_1]+1}^{[Nt_2]} \hspace{-8pt} K^{\prime}_{\infty}(0)  L([Nt_2]) {\varepsilon}_{n}  + 
	\sum_{n=1}^{[Nt_1]} 	K^{\prime}_{\infty}(0)  \big(L([Nt_2])-L([Nt_1])\big) {\varepsilon}_{n}  \Big),
\end{array}\end{equation*}
asymptotically independent, since all other terms are negligible in the limit.
As $\eps_n$'s are independent, it suffices to show that as $N\to\infty$,
\begin{align}\label{asind}
	A_N^{-1} \sum_{n=1}^{[Nt_1]} 	K^{\prime}_{\infty}(0)  \big(L([Nt_2])-L([Nt_1])\big) {\varepsilon}_{n} \overset{\P}{\longrightarrow}0.
\end{align}
Recall that $A_N=N_{2,H_2}L(N)=N^{\frac{1}{2}}H^{\frac{1}{2}}(N)L(N)$.
Since $\lim_{N\to\infty}\big(L([Nt_2])-L([Nt_1])\big)/L(N)=0$ and $N^{-1}_{2,H_2}\sum_{n=1}^{[Nt_1]}\eps_n$ converges weakly to a normal random variable with variance $t_1$, the convergence in probability (\ref{asind}) holds, and this completes the proof.

\section{Proof of Theorem \ref{thm3}} \label{S:thm3}

In the case $\alpha\beta=2$ with $\beta \geq1$ and  $\sum_{j=1}^{\infty}a_{j}^{\alpha/2} H^{1/2}(a_{j}^{-1}) < \infty$, the linear process $X$ defined in \eref{lp} has short memory.  We first give a good $L_2$ approximation for $S_{N}$ by $S_{N,l}$ defined in (\ref{Tl}).

\begin{proposition} \label{prop62} Under the assumptions of Theorem \ref{thm3},
	\begin{align*}
		\lim_{l \rightarrow \infty} \limsup_{N \to \infty} N^{-1} \mathbb{E}[S_{N} - S_{N,l}]^2= 0.
	\end{align*}
\end{proposition}

\begin{proof} Observe that $S_{N}-S_{N,l}=\sum^N_{n=1} \sum^{\infty}_{j=1} \mathscr{P}^{l}_{n,n-j}$ where  
		\begin{align*}
\mathscr{P}^{l}_{n,n-j}=\sum^N_{n=1} \sum^{\infty}_{j=1}\Big( \mathbb{E}\big[K(X_n)|\mathcal{F}_{n-j}\big]-\mathbb{E}\big[K(X_n)|\mathcal{F}_{n-j-1}\big]\Big) - \Big( \mathbb{E}\big[K(X_{n,l})|\mathcal{F}_{n-j}\big]-\mathbb{E}\big[K(X_{n,l})|\mathcal{F}_{n-j-1}\big]\Big)
	\end{align*}
	and $\mathbb{E}[\mathscr{P}^{l}_{n,n-j} \mathscr{P}^{l}_{n{'},n{'}-j{'}}] = 0$ whenever $n-j\neq n{'}-j{'}$. By assumption ({\bf A2}), there exists $m_{0} \in \mathbb{N}$ such that $\prod\limits_{i=1}^{m_{0}} |\phi_{\varepsilon}(a_i u)|$ is less than a constant multiple of $\frac{1}{1+ |u|^{4}}$.

	\noindent
	{\bf Step 1} We estimate $\mathbb{E}\big[\sum\limits^N_{n=1} \sum\limits^{m_0}_{j=1} \mathscr{P}^{l}_{n,n-j}\big]^2$ for $l>m_0$. Clearly,
\begin{align*}
\mathbb{E}\big[\sum^N_{n=1} \sum^{m_0}_{j=1} \mathscr{P}^{l}_{n,n-j}\big]^2
&=\sum^N_{n=1} \sum^{m_0}_{j=1}\sum^N_{m=1} \sum^{m_0}_{i=1}\mathbb{E}\big[ \mathscr{P}^{l}_{n,n-j} \mathscr{P}^{l}_{m,m-i}\big]\\
&=\sum^N_{n=1} \sum^{m_0}_{j=1}\sum^{m_0}_{i=1}\mathbb{E}[ \mathscr{P}^{l}_{n,n-j} \mathscr{P}^{l}_{n-j+i,n-j}] 1_{\{1\leq n-j+i\leq N\}}\\
&\leq \sum^N_{n=1} \sum^{m_0}_{j=1}\sum^{m_0}_{i=1}\mathbb{E}\big[ |\mathscr{P}^{l}_{n,n-j}|^2\big]+ \sum^N_{n=1} \sum^{m_0}_{j=1}\sum^{m_0}_{i=1}\mathbb{E}\big[|\mathscr{P}^{l}_{n-j+i,n-j}|^2\big].
\end{align*} 
Note that $\mathscr{P}^{l}_{n,n-j}$ and $\mathscr{P}^{l}_{n-j+i,n-j}$ have the same law as  $\mathscr{P}^{l}_{1,1-j}$ and $\mathscr{P}^{l}_{1,1-i}$, respectively. Hence,
\begin{align*}
\mathbb{E}\big[\sum^N_{n=1} \sum^{m_0}_{j=1} \mathscr{P}^{l}_{n,n-j}\big]^2
&\leq 2 m_0 N\mathbb{E}\big[K(X_1)-K(X_{1,l})\big]^2.
\end{align*} 
Recall the choice of $m_0$. $X_{1,m_0}=\sum\limits^{m_0}_{j=1}a_n\varepsilon_{1-j}$ has a bounded densify function $f_{m_0}(x)$ and thus 
\begin{align} \label{K1}
&\lim_{l \rightarrow \infty} \limsup_{N \to \infty} N^{-1} \mathbb{E}\big[K(X_1)-K(X_{1,l})\big]^2\nonumber \\
&=\lim_{l \rightarrow \infty} \limsup_{N \to \infty} N^{-1} \mathbb{E}\Big[\int_{\mathbb{R}} |K(x+X_1-X_{1,m_0})-K(x+X_{1,l}-X_{1,m_0})|^2 f_{m_0}(x)\, dx \Big]\nonumber \\
&\leq c_1\lim_{l \rightarrow \infty} \limsup_{N \to \infty} N^{-1}  \mathbb{E}\Big[\int_{\mathbb{R}} |\widehat{K}(u)(e^{-\iota u(X_1-X_{1,m_0})}-e^{-\iota u(X_{1,l}-X_{1,m_0})}|^2 du \Big] \nonumber \\ 
&\leq c_1 \lim_{l \rightarrow \infty} \limsup_{N \to \infty} N^{-1} \int_{\mathbb{R}} |\widehat{K}(u)|^2 \mathbb{E}\big[|(e^{-\iota u X_1}-e^{-\iota u X_{1,l}}|^2  \big] du=0,
\end{align}
where in the last equality we used the dominated convergence theroem and $X_{1,l}\overset{a.s.}{\to} X_{1}$ as $l\to\infty$. 

\noindent
	{\bf Step 2} We estimate $\mathbb{E}\big[\sum\limits^N_{n=1} \sum\limits^{l}_{j=m_0+1} \mathscr{P}^{l}_{n,n-j}\big]^2$ for $l>m_0$. By Plancherel formula,
\begin{align*}
\mathscr{P}^l_{n,n-j}=\frac{1}{2\pi}\int_{\mathbb{R}}\widehat{K}(u)\prod^{j-1}_{k=1}\phi_{\varepsilon}(-a_ku)\big(e^{- \iota u a_{j} \varepsilon_{n-j}} - \phi_{\varepsilon}(-a_{j} u)\big) e^{- \sum\limits_{k=j+1}^{l} \iota u  a_{k} \varepsilon_{n-k}} (e^{- \sum\limits_{k=l+1}^{\infty} \iota u  a_{k} \varepsilon_{n-k}}-1)\, du.
\end{align*}
Hence, $\mathbb{E}\big[\sum\limits^N_{n=1} \sum\limits^{l}_{j=m_0+1} \mathscr{P}^{l}_{n,n-j}\big]^2$ is less than
\begin{align*}
&c_2 \sum^N_{n=1}\sum^{l}_{j=m+1}\sum^l_{i=j}\int_{\mathbb{R}^2}\frac{1}{1+u^4}\frac{1}{1+v^4}\Big|\mathbb{E}\big[(e^{- \iota u a_{j} \varepsilon_{n-j}} - \phi_{\varepsilon}(-a_{j} u))(e^{\iota v a_{i} \varepsilon_{n-j}} - \phi_{\varepsilon}(a_{i} v))\big]\Big| \\
&\qquad\qquad \times \Big|\mathbb{E}\big[(e^{- \sum\limits_{k=l+1}^{\infty} \iota u  a_{k} \varepsilon_{n-k}}-1)(e^{\sum\limits_{k=l+1}^{\infty} \iota v  a_{k+i-j} \varepsilon_{n-k}}-1)\big]\Big|\, du\, dv.
\end{align*}
Note that $\mathbb{E}\big[\big|e^{- \sum\limits_{k=l+1}^{\infty} \iota u  a_{k} \varepsilon_{n-k}}-1\big|^2\big]=4\mathbb{E}\big[\sin^2\big(\frac{1}{2}\sum\limits_{k=l+1}^{\infty}  u  a_{k} \varepsilon_{n-k}\big)\big]
$. So, for any $\alpha'\in (0,\alpha)$ with $\alpha'\beta>1$, by von Bahr-Esseen inequality in \cite{vBE} and the definition of innovations,
\[
\mathbb{E}\big[\big|e^{- \sum\limits_{k=l+1}^{\infty} \iota u  a_{k} \varepsilon_{n-k}}-1\big|^2\big]\leq c_3 \sum^{\infty}_{k=l+1}|ua_k|^{\alpha'}.
\]
Now, by Cauchy-Schwarz inequality and Lemma \ref{lemb}, $\mathbb{E}\big[\sum\limits^N_{n=1} \sum\limits^{l}_{j=m_0+1} \mathscr{P}^{l}_{n,n-j}\big]^2$ is less than  
\begin{align*}
&c_4 \sum^N_{n=1}\sum^{l}_{j=m_0+1}\sum^l_{i=j} \int_{\mathbb{R}^2}\frac{1}{1+u^4}\frac{1}{1+v^4} \Big(\big(|ua_j|^{\frac{\alpha}{2}}H^{\frac{1}{2}}(|ua_j|^{-1})\big)\wedge 1\Big)\Big(\big(|va_i|^{\frac{\alpha}{2}}H^{\frac{1}{2}}(|va_i|^{-1})\big)\wedge 1\Big)\\
&\qquad\qquad \times \big(\sum^{\infty}_{k=l+1}|ua_k|^{\alpha'}\big)^{\frac{1}{2}} \big(\sum^{\infty}_{k=l+1} |ua_{k+i-j}|^{\alpha'}\big)^{\frac{1}{2}} du\, dv\\
&\leq c_4 N  \sum^{\infty}_{k=l+1}|a_k|^{\alpha'}  \Big(\sum^{l}_{j=m+1}\int_{\mathbb{R}}\frac{|u|^{\frac{\alpha'}{2}}}{1+u^4}\Big(\big(|ua_j|^{\frac{\alpha}{2}}H^{\frac{1}{2}}(|ua_j|^{-1})\big)\wedge 1\Big)\, du\Big)^2.
\end{align*}

By the Potter bounds \cite[Theorem 1.5.6]{BGT89}, 
\begin{align*}
&\int_{\mathbb{R}}\frac{|u|^{\frac{\alpha'}{2}}}{1+u^4} \Big(\big(|ua_j|^{\frac{\alpha}{2}}H^{\frac{1}{2}}(|ua_j|^{-1})\big)\wedge 1\Big)\, du \\
&\leq c_5\int_{|u|\leq |a_j|^{-1}}\frac{|u|^{\frac{\alpha'}{2}}}{1+u^4} |ua_j|^{\frac{\alpha}{2}}H^{\frac{1}{2}}(|a_j|^{-1}) du+\int_{|u|>|a_j|^{-1}}\frac{|u|^{\frac{\alpha'}{2}}}{1+u^4}\, du\leq c_6 |a_j|^{\frac{\alpha}{2}}H^{\frac{1}{2}}(|a_j|^{-1}).
\end{align*}
Therefore,
\begin{align} \label{K2}
\mathbb{E}\Big[\sum\limits^N_{n=1} \sum\limits^{l}_{j=m_0+1} \mathscr{P}^{l}_{n,n-j}\Big]^2\leq c_7 N \sum^{\infty}_{k=l+1}|a_k|^{\alpha'} \Big(\sum^{l}_{j=m+1}|a_j|^{\frac{\alpha}{2}}H^{\frac{1}{2}}(|a_j|^{-1})\Big)^2.
\end{align}

\noindent
	{\bf Step 3} We estimate $\mathbb{E}\big[\sum\limits^N_{n=1} \sum\limits^{\infty}_{j=l+1} \mathscr{P}^{l}_{n,n-j}\big]^2$. Observe that 
\begin{align*}
\mathbb{E}\big[\sum^N_{n=1} \sum^{\infty}_{j=l+1} \mathscr{P}^{l}_{n,n-j}\big]^2
&=\sum^N_{n=1} \sum^{\infty}_{j=l+1}\sum^{\infty}_{i=l+1}\mathbb{E}[ \mathscr{P}^{l}_{n,n-j} \mathscr{P}^{l}_{n-j+i,n-j}] 1_{\{1\leq n-j+i\leq N\}}
\end{align*}
and 
\begin{align*}
\mathscr{P}^{l}_{n,n-j}=\frac{1}{2\pi}\int_{\mathbb{R}}\widehat{K}(u)\prod^{j-1}_{k=1}\phi_{\varepsilon}(-a_ku)\big(e^{- \iota u a_{j} \varepsilon_{n-j}} - \phi_{\varepsilon}(-a_{j} u)\big) e^{- \sum\limits_{k=j+1}^{\infty} \iota u  a_{k} \varepsilon_{n-k}} du
\end{align*}
for $j\geq l+1$.  

By assumption ({\bf A2}), Cauchy-Schwartz inequality and Lemma \ref{lemb},
\begin{align*}
&\big|\mathbb{E}[ \mathscr{P}^{l}_{n,n-j} \mathscr{P}^{l}_{n-j+i,n-j}]\big|\\
&\leq c_8 \int_{\mathbb{R}^2}\frac{1}{1+u^4}\frac{1}{1+v^4}\Big|\mathbb{E}\big[\big(e^{- \iota u a_{j} \varepsilon_{n-j}} - \phi_{\varepsilon}(-a_{j} u)\big)\big(e^{\iota v a_{i} \varepsilon_{n-j}} - \phi_{\varepsilon}(a_{i} v)\big]\Big| du\, dv\\
&\leq c_9 \int_{\mathbb{R}}\frac{1}{1+u^4} \Big(\big(|ua_j|^{\frac{\alpha}{2}} H^{\frac{1}{2}}(|ua_j|^{-1})\big)\wedge 1\Big)\, du\int_{\mathbb{R}}\frac{1}{1+v^4} \Big(\big(|va_i|^{\frac{\alpha}{2}}H^{\frac{1}{2}}(|va_i|^{-1})\big)\wedge 1\Big)\, dv.
\end{align*}
Then, using similar arguments as in {\bf Step 2}, we could obtain 
\begin{align} \label{K3}
\mathbb{E}\big[\sum^N_{n=1} \sum^{\infty}_{j=l+1} \mathscr{P}^{l}_{n,n-j}\big]^2
&\leq c_{10}\sum^N_{n=1}\Big(\sum^{\infty}_{j=l+1}|a_j|^{\frac{\alpha}{2}} H^{\frac{1}{2}}(|a_j|^{-1})\Big)^2.
\end{align}

\noindent
{\bf Step 4} Combining \eqref{K1}, \eqref{K2} and \eqref{K3} gives the desired result.
\end{proof}

\medskip
\noindent
{\bf Proof of Theorem \ref{thm3}}:  By $l$-dependence, as $N\to\infty$, $N^{-1/2} S_{N,l} \overset{\mathcal{L}}{\longrightarrow}  N(0,\theta^{2}_{l})$ where $\theta^{2}_{l} = \lim_{N \rightarrow \infty} N^{-1}\text{Var}(S_{N,l})\in (0,\infty)$. Note that
\[
N^{-1} \text{Var}(S_{N,l_1} - S_{N,l_2}) \leq 2 N^{-1} \text{Var}(S_{N} - S_{N,l_1}) + 2 N^{-1} \text{Var}(S_{N} - S_{N,l_2}).
\]
So, by Proposition \ref{prop62}, $\{\theta^{2}_{l}, l =1, 2, \dots \}$ is a Cauchy sequence.
Letting $\theta^{2}=\lim_{l\to\infty} \theta^{2}_{l}$, we see that $N^{-1/2} S_{N}$ converges weakly to $N(0,\theta^2)$ as $N\to\infty$.

For each $m\in\mathbb{N}$ and $0=t_0<t_1<\cdots<t_m<\infty$, 
\[
N^{-1/2}(S_{[Nt_m],l}-S_{[Nt_{m-1}],l}),N^{-1/2}(S_{[Nt_{m-1}],l}-S_{[Nt_{m-2}],l})\cdots, N^{-1/2}(S_{[Nt_1],l}-S_{[Nt_{0}],l})
\]
are independent whenever $N$ is large. So, by Proposition \ref{prop62} and \cite[Theorem 4.2]{Billingsley}, as $N\to\infty$,
\[
\big\{ N^{-1/2}S_{[Nt]}: t\geq 0\big\} \overset{\rm f.d.d.}{\longrightarrow} \big\{ |\theta| W_t:\; t\geq 0 \big\}.
\]

\bigskip
\noindent
{\bf Acknowledgements.} F. Xu is partially supported by the National Natural Science Foundation of China (Grant No. 12371156). J.~Yu is supported by National Key R\&D Program of China (No.~2021YFA1002700) and National Natural Science Foundation of China (Grant No.~12101238 and 12271010).

$\begin{array}{cc}

\begin{minipage}[t]{1\textwidth}

{\bf Yudan Xiong}

\medskip
School of Statistics, East China Normal University, Shanghai 200262, China\\ [3pt]
\texttt{52284404007@stu.ecnu.edu.cn}

\medskip\medskip

{\bf Fangjun Xu}

\medskip

KLATASDS-MOE, School of Statistics, East China Normal University, Shanghai, 200062, China 
\\ [3pt]
NYU-ECNU Institute of Mathematical Sciences at NYU Shanghai, Shanghai, 200062, China\\ [3pt]
\texttt{fjxu@finance.ecnu.edu.cn}

\medskip\medskip

{\bf Jinjiong Yu}

\medskip

KLATASDS-MOE, School of Statistics, East China Normal University, Shanghai, 200062, China \\ [3pt]
NYU-ECNU Institute of Mathematical Sciences at NYU Shanghai, Shanghai, 200062, China\\ [3pt]
\texttt{jjyu@sfs.ecnu.edu.cn}

\end{minipage}

\hfill

\end{array}$

\end{document}